\documentclass[12pt, reqno]{amsart}
\usepackage{amsmath, amsthm, amscd, amsfonts, amssymb, graphicx, color, float, enumerate, graphicx}
\usepackage[bookmarksnumbered, colorlinks, plainpages]{hyperref}
\input{mathrsfs.sty}
\hypersetup{colorlinks=true,linkcolor=red, anchorcolor=green, citecolor=cyan, urlcolor=red, filecolor=magenta, pdftoolbar=true}

\newtheorem{theorem}{Theorem}[section]
\newtheorem{lemma}[theorem]{Lemma}
\newtheorem{proposition}[theorem]{Proposition}
\newtheorem{corollary}[theorem]{Corollary}
\theoremstyle{definition}
\newtheorem{definition}[theorem]{Definition}
\newtheorem{example}[theorem]{Example}

\newtheorem{problem}[theorem]{Problem}
\theoremstyle{remark}
\newtheorem{remark}[theorem]{Remark}
\numberwithin{equation}{section}

\begin{document}

\title[Schatten norms on Hilbert $C^*$-modules via pure states]{Schatten norms on Hilbert $C^*$-modules via pure states}

\author[S. Abedi, M. S. Moslehian]{Sajjad Abedi \MakeLowercase{and} Mohammad Sal Moslehian}

\address{Department of Pure Mathematics, Faculty of Mathematical Sciences, Ferdowsi University of Mashhad, P. O. Box 1159, Mashhad 91775, Iran}
\email{sajjad.abedi@mail.um.ac.ir; sajjadabedi2022@gmail.com}
\email{moslehian@um.ac.ir; moslehian@yahoo.com}

\subjclass{Mathematics Subject Classification}


\begin{abstract}
Let $(\mathscr{E}, \langle \cdot, \cdot\rangle)$ be a Hilbert $C^*$-module over a $C^*$-algebra $\mathfrak{A}$. The space of adjointable operators on $\mathscr{E}$ is denoted by $\mathcal{L}\left(\mathscr{E}\right)$. The sets of all states and pure states on $\mathfrak{A}$ are denoted by $\mathcal{S}\left(\mathfrak{A}\right)$ and $\mathcal{P}\left( \mathfrak{A}\right)$, respectively. For  $\tau\in\mathcal{S}\left( \mathfrak{A}\right)$, let us define $\mathcal{N}^{\mathscr{E}}_{\tau}:=\left\lbrace x\in\mathscr{E}:\tau\left( \left\langle x,x\right\rangle\right)=0 \right\rbrace$. The Hilbert completion of ${\mathscr{E}}/{\mathcal{N}^{\mathscr{E}}_{\tau}}$ is denoted by $\mathscr{E}_{\tau}$.  For $T\in\mathcal{L}(\mathscr{E})$, the operator $T_{\mathscr{E}_\tau}\in\mathbb{B}\left( \mathscr{E}_\tau\right)$, is defined by $
T_{\mathscr{E}_\tau}\left(x+\mathcal{N}^{\mathscr{E}}_{\tau}\right)=Tx+\mathcal{N}^{\mathscr{E}}_{\tau}$. In this paper, we  show that $\mathscr{E}_{\tau}={\mathscr{E}}/{\mathcal{N}^{\mathscr{E}}_{\tau}}$ when $\mathfrak{A}$ either is a $C^*$-algebra of compact operators or is commutative. 

We introduce a quantity in the context of Hilbert $C^*$-modules, denoted by 
$\pi^{\mathscr{E}}_k(\cdot)$ for $k\geq1$. We prove that $\pi^{\mathscr{E}}_k(T)\leq\sup_{\tau\in\mathcal{P}\left( \mathfrak{A}\right)}\left\|T_{\mathscr{E}_\tau}\right\|_{\left(k\right)}\leq{\pi}^{\mathscr{E}^{\sharp}}_k(T_{\mathscr{E}^{\sharp}})$ for every $T\in\mathcal{L}\left(\mathscr{E}\right)$,  where the space $\mathscr{E}^{\sharp}$ is constructed as the extension of $\mathscr{E}$ by the embedding of $\mathfrak{A}$ into its enveloping von Neumann algebra $\mathfrak{A}^{**}$. Thus, the norm $\pi^{\mathscr{E}}_k(\cdot)$ preserves the Schatten properties, 
as in Hilbert spaces, particularly when 
$\pi^{\mathscr{E}}_k(T) = {\pi}^{\mathscr{E}^{\sharp}}_k(T_{\mathscr{E}^{\sharp}}) 
\quad \text{for all } T \in \mathcal{L}(\mathscr{E})$. This equality holds for $\mathscr{E}$ over commutative $C^*$-algebras with a frame. Moreover, it holds for $\mathscr{E}$ over $C^*$-algebras of compact operators.
\end{abstract}
\maketitle

\section{Introduction and preliminaries}
 Let $\left(\mathcal{X}_{\lambda}\right)_{\lambda \in \Lambda}$ be a family of Banach algebras. We consider the \emph{direct product}  $\prod_{\lambda \in \Lambda}\mathcal{X}_{\lambda}$ consisting of all $(x_{\lambda})$ such that $\left\|\left(x_\lambda\right) \right\|:=\sup_{\lambda\in\Lambda}\left\|x_\lambda \right\|<\infty$. This is a Banach algebra under pointwise operations. Moreover, the \emph{$c_0$-direct sum} $\Sigma_{\lambda\in\Lambda}\mathcal{X}_{\lambda}$ is the set of all $(x_\lambda)\in\prod_{\lambda \in \Lambda}\mathcal{X}_{\lambda}$ such that for each $\varepsilon>0$ there exists a finite subset $\mathcal{F}$ of $\Lambda$ such that $\left\|x_{\lambda}\right\|<\varepsilon$ whenever $\lambda\in\Lambda-\mathcal{F}$. The $c_0$-direct sum $\Sigma_{\lambda \in \Lambda}\mathcal{X}_{\lambda}$ is a closed ideal of $\prod_{\lambda \in \Lambda}\mathcal{X}_{\lambda}$. 

We also recall the \emph{$\ell^1$-direct sum}  $\oplus^{\ell^1}_{\lambda \in \Lambda}\mathcal{X}_{\lambda}$ which consists of all $(x_{\lambda})$ such that $\left\|\left(x_\lambda\right) \right\|_{\ell^1}:=\sup\sum_{\lambda\in\mathcal{F}}\left\|x_\lambda \right\| < \infty$, where the supremum is taken over all finite subsets $\mathcal{F}$ of $\Lambda$. It is straightforward to verify that
	\begin{equation}\label{duality}
		\left(\Sigma_{\lambda \in \Lambda}\mathcal{X}_{\lambda} \right)^* \cong\oplus^{\ell^1}_{\lambda \in \Lambda}\mathcal{X}^*_{\lambda}\quad\text{and}\quad\left(\oplus^{\ell^1}_{\lambda \in \Lambda}\mathcal{X}_{\lambda} \right)^* \cong\prod_{\lambda \in \Lambda}\mathcal{X}^*_{\lambda}.
\end{equation}

Suppose that  $\left(\mathfrak{A}_{\lambda}\right)_{\lambda \in \Lambda}$ is a family of $C^*$-algebras. Then, both the  direct product  $\prod_{\lambda \in \Lambda}\mathfrak{A}_{\lambda}$ and the $c_0$-direct sum $\Sigma_{\lambda\in\Lambda}\mathfrak{A}_{\lambda}$ are $C^*$-algebras.  For example, a $C^*$-algebra of compact operators is of the form
 $\Sigma_{\lambda\in\Lambda}\mathbb{K}\left(\mathscr{H}_{\lambda}\right)$, a $c_0$-direct sum of $C^*$-algebras $\mathbb{K}\left(\mathscr{H}_{\lambda}\right)$ of all compact
 operators acting on some Hilbert space $\mathscr{H}_{\lambda}$; see  \cite[Theorem 1.4.5]{Arveson}.
 
 Assume that $\mathfrak{A}$ is a $C^*$-algebra.
 The unit of a unital $C^*$-algebra $\mathfrak{A}$ is denoted by $1_{\mathfrak{A}}$.
 A positive linear functional on $\mathfrak{A}$ with norm one is said to be a \emph{state}.  A state $\tau$ is called \emph{pure} if for each positive linear functional $\rho$ on $\mathfrak{A}$ with $\rho\leq\tau$, there exists $0\leq t\leq1$ such that $\rho=t\tau$.  The set of all states and pure states on $\mathfrak{A}$ are denoted by $\mathcal{S}\left(\mathfrak{A}\right)$ and $\mathcal{P}\left( \mathfrak{A}\right)$, respectively. If $\mathfrak{A}$ is commutative, we observe that $\mathcal{P}\left(\mathfrak{A}\right)=\Omega\left(\mathfrak{A}\right)$, where $\Omega\left( \mathfrak{A}\right)$ is the character space of $\mathfrak{A}$ equipped with the weak$\rm{^*}$ topology.

 Throughout this paper, suppose that $\mathscr{E}$ be a (right) \emph{Hilbert $\mathfrak{A}$-module} with the inner product $\langle \cdot, \cdot\rangle_{\mathscr{E}}:\mathscr{E}\times\mathscr{E}\rightarrow\mathfrak{A}$. We denote this inner product by $\langle \cdot,\cdot\rangle$ if there is no ambiguity.  For example, a $C^*$-algebra $\mathfrak{A}$ is a Hilbert $\mathfrak{A}$-module if we define $\left\langle a,b\right\rangle=a^*b$ for all $a,b\in\mathfrak{A}$. In addition, if $\mathscr{H}$ and $\mathscr{K}$ are Hilbert spaces, then the space $\mathbb{B}\left(\mathscr{H},\mathscr{K}\right)$ is a Hilbert $\mathbb{B}\left(\mathscr{H}\right)$-module under the inner product $\left\langle T,S\right\rangle=T^*S$ for $T,S\in\mathbb{B}\left(\mathscr{H},\mathscr{K}\right)$. 	Set $\left|x\right|:=\left\langle x, x \right\rangle^{1/2}$ for each $x\in\mathscr{E}$. 
 Let $\left(\mathscr{E}_{\lambda}\right)_{\lambda \in \Lambda}$ be a family of Hilbert $\mathfrak{A}$-modules. The direct sum $\oplus_{\lambda\in\Lambda}\mathscr{E}_{\lambda}$ is the Hilbert $\mathfrak{A}$-module consisting of all $(x_\lambda)$ such that $\sum_{\lambda \in \Lambda} \left\langle x_\lambda,x_\lambda\right\rangle$ converges in $\mathfrak{A}$ under the norm-topology. This $\mathfrak{A}$-module equipped with the inner product $$\left\langle(x_\lambda),(y_\lambda)\right\rangle=\sum_{\lambda \in \Lambda} \left\langle x_\lambda,y_\lambda\right\rangle, \qquad (x_\lambda),(y_\lambda)\in\oplus_{\lambda\in\Lambda}\mathscr{E}_{\lambda}$$
 is a Hilbert $C^*$-module. The direct sum of $n$-copies of  $\mathscr{E}$ is denoted by $\mathit{l}_n^2(\mathscr{E})$.
 
 Let $\tau\in\mathcal{S}\left(\mathfrak{A}\right)$. Set $\mathcal{N}^{\mathscr{E}}_{\tau}:=\left\lbrace x\in\mathscr{E}:\tau\left( \left\langle x,x\right\rangle_{\mathscr{E}}\right) =0 \right\rbrace.$
 Consider the scalar-valued inner product $\langle\cdot,\cdot\rangle:\mathscr{E}/\mathcal{N}^{\mathscr{E}}_{\tau}\times\mathscr{E}/\mathcal{N}^{\mathscr{E}}_{\tau}\rightarrow\mathbb{C}$ given by $$\left\langle x+\mathcal{N}^{\mathscr{E}}_{\tau},y+\mathcal{N}^{\mathscr{E}}_{\tau}\right\rangle_{\mathscr{E}_{\tau}}=\tau\left( \left\langle x,y\right\rangle_{\mathscr{E}}\right).$$
 The Hilbert completion of ${\mathscr{E}}/{\mathcal{N}^{\mathscr{E}}_{\tau}}$ is denoted by $\mathscr{E}_{\tau}$. If $\mathfrak{A}$ is considered as a Hilbert $C^*$-module over itself and $\tau\in\mathcal{P}\left(\mathfrak{A}\right)$, Kadison \cite{Kadison} used his transitivity theorem to establish that $\mathfrak{A}_{\tau}={\mathfrak{A}}/{\mathcal{N}^{\mathfrak{A}}_{\tau}}$; in other words, ${\mathfrak{A}}/{\mathcal{N}^{\mathfrak{A}}_{\tau}}$ is already complete.
 
 Let $\mathscr{E}$ and $\mathscr{F}$ be Hilbert $\mathfrak{A}$-modules. A map $T: \mathscr{E}\rightarrow \mathscr{F}$ is said to be \emph{adjointable} if there exists a map $T^*: \mathscr{F}\rightarrow \mathscr{E}$ such that $\left\langle Tx,y \right\rangle=\left\langle x,T^*y \right\rangle$ for all $x \in \mathscr{E}$ and $y \in \mathscr{F}$. The space of all adjointable maps from $\mathscr{E}$ to $\mathscr{F}$ is denoted by $\mathcal{L}(\mathscr{E},\mathscr{F})$. For $x \in \mathscr{E}$ and $y \in \mathscr{F}$, we define ${\theta}_{y,x}:\mathscr{E} \rightarrow \mathscr{F}$ by ${\theta}_{y,x}(z)\mapsto y\left\langle x,z\right\rangle$ for $z \in \mathscr{E}$. It follows immediately that $\left( {\theta}_{y,x}\right)^*={\theta}_{x,y}$. So ${\theta}_{y,x} \in \mathcal{L}(\mathscr{E},\mathscr{F})$. The norm-closed linear span of $\left\lbrace {\theta}_{y,x}: x \in \mathscr{E}, y \in \mathscr{F} \right\rbrace $ is denoted by $\mathcal{K}(\mathscr{E},\mathscr{F})$. We use the notations $\mathcal{L}(\mathscr{E})$ and $\mathcal{K}(\mathscr{E})$
for $\mathcal{L}(\mathscr{E},\mathscr{E})$ and $\mathcal{K}(\mathscr{E},\mathscr{E})$, respectively. Following \cite[p. 8]{lance}, they are  $C^*$-algebras. 
 
Let us recall some facts about representations of Hilbert $C^*$-modules from \cite{representation}. Suppose that $\mathscr{E}$ is a Hilbert $\mathfrak{A}$-module and $\mathscr{F}$ is a Hilbert $\mathfrak{B}$-module. Let $\phi:\mathfrak{A}\rightarrow\mathfrak{B}$ be a $*$-homomorphism. 
 A map $\Phi:\mathscr{E}\rightarrow\mathscr{F}$
is called a \emph{$\phi$-morphism} if for $x,y\in\mathscr{E}$, we have $\left\langle\Phi(x),\Phi(y)\right\rangle=\phi\left(\left\langle x,y\right\rangle\right)$. A $\phi$-morphism $\Phi:\mathscr{E}\rightarrow\mathbb{B}\left(\mathscr{H},\mathscr{K}\right) $, where $\phi:\mathfrak{A}\rightarrow\mathbb{B}\left(\mathscr{H}\right)$ is a representation of $\mathfrak{A}$, is said to be a \emph{representation of the Hilbert $C^*$-module} $\mathscr{E}$. Let $\mathscr{H}_1\leq\mathscr{H}$ and $\mathscr{K}_1\leq\mathscr{K}$ be closed subspaces. A pair $\left(\mathscr{H}_1,\mathscr{K}_1\right)$ is called \emph{$\Phi$-invariant} if $\Phi\left(\mathscr{E}\right)\left(\mathscr{H}_1 \right)\subseteq \mathscr{K}_1 $ and $\Phi\left(\mathscr{E}\right)^*\left(\mathscr{K}_1\right)\subseteq\mathscr{H}_1$. The representation $\Phi$ is called \emph{irreducible} if $\left(0,0\right)$ and $\left(\mathscr{H},\mathscr{K}\right)$ are the only $\Phi$-invariant pairs.

 
 Let $\tau\in\mathcal{S}\left(\mathfrak{A}\right)$. Set $\mathcal{N}^{\mathscr{E}}_{\tau}:=\left\lbrace x\in\mathscr{E}:\tau\left( \left\langle x,x\right\rangle_{\mathscr{E}}\right) =0 \right\rbrace.$
 Consider the scalar-valued inner product $\langle\cdot,\cdot\rangle:\mathscr{E}/\mathcal{N}^{\mathscr{E}}_{\tau}\times\mathscr{E}/\mathcal{N}^{\mathscr{E}}_{\tau}\rightarrow\mathbb{C}$ given by $$\left\langle x+\mathcal{N}^{\mathscr{E}}_{\tau},y+\mathcal{N}^{\mathscr{E}}_{\tau}\right\rangle_{\mathscr{E}_{\tau}}=\tau\left( \left\langle x,y\right\rangle_{\mathscr{E}}\right).$$
 The Hilbert completion of ${\mathscr{E}}/{\mathcal{N}^{\mathscr{E}}_{\tau}}$ is denoted by $\mathscr{E}_{\tau}$. If $\mathfrak{A}$ is considered as a Hilbert $C^*$-module over itself and $\tau\in\mathcal{P}\left(\mathfrak{A}\right)$, Kadison \cite{Kadison} used his transitivity theorem to establish that $\mathfrak{A}_{\tau}={\mathfrak{A}}/{\mathcal{N}^{\mathfrak{A}}_{\tau}}$; in other words, ${\mathfrak{A}}/{\mathcal{N}^{\mathfrak{A}}_{\tau}}$ is already complete.
 
Inspired by \cite{lance}, we introduce the notion of \emph{interior tensor product}. Suppose that $\mathscr{E}$ is a Hilbert $\mathfrak{A}$-module and $\mathscr{F}$ is a Hilbert $\mathfrak{B}$-module. Let $\phi: \mathfrak{A}\rightarrow\mathcal{L}(\mathscr{F})$ be a $*$-homomorphism of $C^*$-algebras. The algebraic tensor product $\mathscr{E}\otimes_{alg}\mathscr{F}$ is a right $\mathfrak{B}$-module by setting $\left(x\otimes y\right)b=x\otimes yb $. The subspace of $\mathscr{E}\otimes_{alg}\mathscr{F}$ generated by elements of the form $xa\otimes y-x\otimes \phi(a)y$, where $x \in \mathscr{E}$, $y \in \mathscr{F}$ and $a\in\mathfrak{A}$, is denoted by $\mathcal{M}_{\phi}$. Letting
\begin{align}\label{msm1}
	\left\langle x_1\otimes y_1,x_2\otimes y_2 \right\rangle_{\phi}:=\left\langle y_1,\phi(\left\langle x_1,x_2\right\rangle)y_2 \right\rangle_{\mathscr{F}},
\end{align}
one gets a semi-inner product on $\mathscr{E}\otimes_{alg}\mathscr{F}$ by extending through $\mathfrak{B}$-linearity. Drawing inspiration from \cite[p. 41]{lance}, $$\mathcal{M}_{\phi}=\left\lbrace u\in \mathscr{E}\otimes_{alg}\mathscr{F}: \left\langle u,u\right\rangle_{\phi}=0\right\rbrace.$$ Consequently, an inner product can be defined on $\left( \mathscr{E}\otimes_{alg}\mathscr{F}\right) /{\mathcal{M}_{\phi}}$. The completion of this inner product $\mathfrak{B}$-module forms a Hilbert $\mathfrak{B}$-module, known as the \emph{interior tensor product} $\mathscr{E}\otimes_{\phi}\mathscr{F}$. 
	Moreover, the map $\mathcal{L}\left(\mathscr{E}\right)\rightarrow \mathcal{L}\left(\mathscr{E}\otimes_{\phi}\mathscr{F}\right)$ defined by $T\mapsto T_{\mathscr{E}\otimes_{\phi}\mathscr{F}}$, is a $*$-homomorphism, where 
\begin{equation}\label{T}
	\left(T_{\mathscr{E}\otimes_{\phi}\mathscr{F}}\right)\left(x\otimes y+\mathcal{M}_{\phi}\right)=Tx\otimes y+\mathcal{M}_{\phi}\quad\left(T\in\mathcal{L}\left(\mathscr{E}\right),x\in\mathscr{E}, y\in\mathscr{F}\right).
\end{equation}
  In the first section, we give a variant of Kadison's result $\mathfrak{A}_{\tau}={\mathfrak{A}}/{\mathcal{N}^{\mathfrak{A}}_{\tau}}$. This is a consequence of the well-known Kadison transitivity theorem \cite{Kadison}, applied to specific classes of Hilbert $C^*$-modules. We use the notion of interior tensor product to introduce a helpful Hilbert space and a Hilbert $C^*$-module as follows:
  
(i) For $\tau\in\mathcal{S}\left(\mathfrak{A}\right)$, the $*$-representation $\phi_{\tau}:\mathfrak{A}\rightarrow\mathbb{B}\left(\mathfrak{A}_{\tau}\right)$ is defined by $$\phi_{\tau}(a)\left(b+\mathcal{N}^{\mathfrak{A}}_{\tau}\right)=ab+\mathcal{N}^{\mathfrak{A}}_{\tau}.$$
Set $\mathcal{E}_{\tau}:=\mathscr{E}\otimes_{\phi_{\tau}}\mathfrak{A}_{\tau}$. It follows from \eqref{msm1} that $\mathcal{E}_{\tau}$ is indeed a Hilbert space. Inspired by \cite[p. 10]{sheide}, we can construct the representation $\Phi_{\tau}:\mathscr{E}\rightarrow\mathbb{B}\left(\mathfrak{A}_{\tau},\mathcal{E}_{\tau}\right)$ by
 \begin{align}
 	\Phi_{\tau}(x)h=x\otimes h+\mathcal{M}_{\phi_{\tau}}\quad\left(x\in\mathscr{E}, h\in\mathfrak{A}_{\tau}\right). 
 \end{align}
For example, let $\mathscr{E}=\mathscr{H}$. Note that the only state on $\mathbb{C}$ is the identity function. Thus, $\mathcal{E}_{\tau}=\mathscr{H}$ and the representation $\Phi_{\tau}:\mathscr{H}\rightarrow\mathbb{B}\left(\mathbb{C},\mathscr{H}\right)$ is defined by $\Phi_{\tau}(x)\alpha=\alpha x$, where $x\in\mathscr{H}$ and $\alpha\in\mathbb{C}$. 

We show that $\mathscr{E}_{\tau}$ is a closed subspace of $\mathcal{E}_{\tau}$. In the case when $\tau\in\mathcal{P}\left(\mathfrak{A}\right)$, we have $\mathcal{E}_{\tau}\cong\mathscr{E}_{\tau}$ and
 $\Phi_{\tau}$ is irreducible.
 Consider a $\phi$-morphism $\Phi:\mathscr{E}\rightarrow\mathscr{F}$. We infer that $\Phi\left(\mathscr{E}\right)$ is a Hilbert $\phi(\mathfrak{A})$-module. Thus, it is closed in $\mathscr{F}$. By using this fact, we derive that 
 $\mathscr{E}_{\tau}={\mathscr{E}}/{\mathcal{N}^{\mathscr{E}}_{\tau}}$ in the case where $\mathfrak{A}$ is commutative and $\tau\in\mathcal{P}\left(\mathfrak{A}\right)$.
  
 (ii) Let $\mathscr{E}$ be a Hilbert $\mathfrak{A}$-module. The inclusion $\iota:\mathfrak{A}\hookrightarrow\mathfrak{A}^{**}$, where $\mathfrak{A}^{**}$ represents the bidual of $\mathfrak{A}$, can be viewed as a $*$-homomorphism. We set $\mathscr{E}^{\sharp}:=\mathscr{E}\otimes_{\iota}\mathfrak{A}^{**}$ and treat it as a Hilbert $C^*$-module. This module is commonly referred as the \emph{extension of $\mathscr{E}$ by the algebra $\mathfrak{A}^{**}$}. For instance, the extension of $\mathfrak{A}$ by the algebra $\mathfrak{A}^{**}$ is $\mathfrak{A}^{**}$. 
 
 For $\tau\in\mathcal{P}\left(\mathfrak{A}\right)$, we demonstrate that there exists a minimal projection $p_{\tau}\in\mathfrak{A}^{**}$ such that 
 $\mathcal{N}^{\mathscr{E}}_{\tau}=\left\lbrace x\in\mathscr{E}: xp_{\tau}=0 \right\rbrace$.
 Moreover, the map $\mathscr{E}/\mathcal{N}^{\mathscr{E}}_{\tau}\rightarrow\mathscr{E}p_{\tau}$ defined by $x+\mathcal{N}^{\mathscr{E}}_{\tau}\mapsto xp_{\tau}$ is an isometric isomorphism. This result yields 
 $\mathscr{E}_{\tau}={\mathscr{E}}/{\mathcal{N}^{\mathscr{E}}_{\tau}}$, whenever $\mathfrak{A}=\Sigma_{\lambda\in\Lambda}\mathbb{K}\left(\mathscr{H}_{\lambda}\right)$. 
 
In the second section,   we introduce two quantities corresponding to Hilbert space Schatten $k$-norms for $k\geq1$ in the context of Hilbert $C^*$-modules. In \cite[Proposition 3.4(1)]{sajjad}, it is presented a new version of the power-norms $\left(\mu_{2,n}:n\in\mathbb{N}\right)$, introduced by Dales \cite[p. 64]{2012}, appropriate for a Hilbert $\mathfrak{A}$-module $\mathscr{E}$ as follows:
 \begin{equation}\label{mue}
 	\mu^{\mathscr{E}}_{n}(x_1,\dots,x_n)=\sup_{y\in\mathscr{E}_{[1]}}\left\|\left( \sum_{i=1}^{n}\left|\left\langle x_i,y\right\rangle\right|^2\right)^{1/2}\right\|,
 \end{equation}
 where $\mathscr{E}_{[1]}$ is the norm-closed unit ball of $\mathscr{E}$; see also \cite{sajjad3}. Moreover, inspired by the proofs of \cite[Proposition 3.6]{sajjad} and \cite[Lemma 3.7]{sajjad},
 \begin{equation}\label{Proposition 3.6}
 	\mu^{\mathscr{E}}_{n}(x_1,\dots,x_n)=\min\left\{\lambda>0: \sum_{i=1}^{n} \left\langle y,x_i\right\rangle\left\langle x_i,y\right\rangle\leq\lambda^2\left\langle y,y\right\rangle\textrm{~for~all~} y\in\mathscr{E}\right\};
 \end{equation}
 \begin{equation}\label{Lemma3.7}
 	\mu^{\mathscr{E}}_{n}(x_1,\dots,x_n)=\sup\left\lbrace\left\| \sum_{i=1}^{n}x_ia_i\right\| : a_1,\dots,a_n\in\mathfrak{A},\left\|\sum_{i=1}^{n}a_i^*a_i \right\|\leq 1\right\rbrace;
 \end{equation}
 \begin{equation}\label{Proposition 3.4(1)}
 	\mu^{\mathscr{E}}_{n}(x_1a_1,\dots,x_na_n)\leq\max_{1\leq i\leq n}\left\|a_i\right\|\mu^{\mathscr{E}}_{n}(x_1,\dots,x_n)\quad\left( a_1,\dots,a_n\in\mathfrak{A}\right).
 \end{equation}
 We refer the readers to \cite{DALMOS, bbb, hilbert, 2017} for more information about power-norms and multi-norms. Now, we define two quantities as follows: 
\begin{align*}
		&\pi^{\mathscr{E}}_k(T): = \sup\left\lbrace \left\|\sum_{i=1}^n \left\langle\left|T\right|^kx_i,x_i\right\rangle\right\|^{1/k}: \left( x_1,\ldots,x_n\right) \in\left( \mathscr{E}^n,\mu^{\mathscr{E}}_n\right)_{[1]}, n\in\mathbb{N}\right\rbrace;\\&\left\|T\right\|_{\left[ k\right]}:=\sup_{\tau\in\mathcal{P}\left( \mathfrak{A}\right)}\left\|T_{\mathscr{E}_\tau}\right\|_{\left(k\right)}\tag{$T\in\mathcal{L}(\mathscr{E})$ and $\left\|T_{\mathscr{E}_\tau}\right\|_{\left( k\right) }=\left(\mathrm {Tr}~ \left| T_{\mathscr{E}_\tau}\right| ^{k}\right)^{\frac{1}{k}}$}, 
	\end{align*}
where $T_{\mathscr{E}_\tau}\in\mathbb{B}\left( \mathscr{E}_\tau\right)$, according to \eqref{T}, is defined by $
	T_{\mathscr{E}_\tau}\left(x+\mathcal{N}^{\mathscr{E}}_{\tau}\right)=Tx+\mathcal{N}^{\mathscr{E}}_{\tau}$.

Jaegermann \cite[Proposition 10.1]{Jaegermann} proved that $\pi^{\mathscr{E}}_k(T)=\left\|T\right\|_{\left( k\right)}$ for all $T\in\mathbb{B}(\mathscr{H})$; see also our Proposition \ref{Jaegermann}. These quantities, along with their properties, provide new insights into the realm of $C^*$-algebras.

The set of adjointable operators with $	\pi^{\mathscr{E}}_k(T) < \infty$ is denoted by ${\Pi}_k(\mathscr{E})$.
We will discuss in Remark \ref{remark} that $\left({\Pi}_{k}(\mathscr{E}),{\pi}^{\mathscr{E}}_{k}(\cdot)\right) $ is a Banach space when $\mathfrak{A}$ is commutative. In general, we demonstrate that it is a Banach space for $k=1,2$ and it is unknown whether this holds for other values of $k$. Furthermore, most Schatten properties may not hold for $\left({\Pi}_{k}(\mathscr{E}),{\pi}^{\mathscr{E}}_{k}(\cdot)\right)$. Our aim is to precisely explore the Schatten properties of $\left({\Pi}_{k}(\mathscr{E}),{\pi}^{\mathscr{E}}_{k}(\cdot)\right) $.
 We prove that 	$\mu^{\mathscr{E}}_{n}(x_1,\dots,x_n)=\mu^{\mathscr{E}^{\sharp}}_{n}(x_1,\dots,x_n)$ for $x_1,\dots,x_n\in\mathscr{E}$. Therefore, $\pi^{\mathscr{E}}_k(T)\leq\pi^{\mathscr{E}^{\sharp}}_k(T_{\mathscr{E}^{\sharp}})$ for every $T\in\mathcal{L}\left(\mathscr{E}\right)$, where the operator $T_{\mathscr{E}^{\sharp}}$ has been constructed at \eqref{T} by considering $\mathscr{E}^{\sharp}=\mathscr{E}\otimes_{\iota}\mathfrak{A}^{**}$.
 
Suppose that $\tau\in\mathcal{S}(\mathfrak{A})$ and $k\geq1$. We show that $\tau(a)^k\leq\tau(a^k)$ for each $a\geq 0$, and infer that $\pi^{\mathscr{E}}_k(T)\leq\left\|T\right\|_{\left[k\right]}\leq{\pi}^{\mathscr{E}^{\sharp}}_k(T_{\mathscr{E}^{\sharp}})$. Moreover, we establish that if $\mathscr{E}$ is a Hilbert $C^*$-module over a commutative $C^*$-algebra with a frame or if $\mathscr{E}$ is a Hilbert $C^*$-module over a $C^*$-algebra of compact operators, then $\pi^{\mathscr{E}}_k(T)=\left\|T\right\|_{\left[k\right]}={\pi}^{\mathscr{E}^{\sharp}}_k(T_{\mathscr{E}^{\sharp}})$ is valid. We demonstrate that $\pi^{\mathscr{E}}_k(\cdot)$ enjoys most Schatten properties that Kittaneh \cite{Kittaneh5} has investigated in the framework of Hilbert spaces.
 
 \section{when is ${\mathscr{E}}/{\mathcal{N}^{\mathscr{E}}_{\tau}}$ complete?}
As mentioned earlier, it is known that $\mathfrak{A}_{\tau}={\mathfrak{A}}/{\mathcal{N}^{\mathfrak{A}}_{\tau}}$; see \cite[Corollary 1]{Kadison}. Following, we attempt to identify certain classes of Hilbert $C^*$-modules $\mathscr{E}$ satisfying $\mathscr{E}_{\tau}={\mathscr{E}}/{\mathcal{N}^{\mathscr{E}}_{\tau}}$.
 \begin{proposition}\label{oplus}
 	Let $\tau\in\mathcal{P}\left(\mathfrak{A}\right)$, let $\left(\mathscr{E}^{\lambda}\right)_{\lambda \in \Lambda}$ be a family of Hilbert $\mathfrak{A}$-modules such that $\mathscr{E}^{\lambda}_{\tau}={\mathscr{E}}^{\lambda}/{\mathcal{N}^{\mathscr{E}^{\lambda}}_{\tau}}$, and let $\mathscr{E}:=\oplus_{\lambda\in\Lambda}\mathscr{E}^{\lambda}$. Then 	$\mathscr{E}_{\tau}={\mathscr{E}}/{\mathcal{N}^{\mathscr{E}}_{\tau}}$.
 \end{proposition}
 \begin{proof}
 	Define the map $\mathscr{E}/{\mathcal{N}^{\mathscr{E}}_{\tau}}\rightarrow\oplus_{\lambda\in\Lambda}\mathscr{E}^{\lambda}_{\tau} $ by
 	\begin{equation}
 		\left(x_\lambda\right)+{\mathcal{N}^{\mathscr{E}}_{\tau}}\mapsto \left(x_\lambda+\mathcal{N}^{\mathscr{E}^{\lambda}}_{\tau}\right)_{\lambda \in \Lambda}.
 	\end{equation}
 	Since
 	\begin{align*}
 		\left\|\left(x_\lambda\right)+{\mathcal{N}^{\mathscr{E}}_{\tau}}\right\|^2&= \tau\left(\left\langle\left(x_\lambda\right),\left(x_\lambda\right)\right\rangle \right)=\sum_{\lambda \in \Lambda} \tau\left( \left\langle x_\lambda,x_\lambda\right\rangle_{\mathscr{E}^{\lambda}}\right)\\&=\sum_{\lambda} \left\|x_\lambda+\mathcal{N}^{\mathscr{E}^{\lambda}}_{\tau}\right\|^2=\left\|\left(x_\lambda+\mathcal{N}^{\mathscr{E}^{\lambda}}_{\tau}\right)_{\lambda \in \Lambda}\right\|^2, 
 	\end{align*}
 	the above map is isometric. By the assumption, $\mathscr{E}^{\lambda}_{\tau}=\mathscr{E}^{\lambda}/{\mathcal{N}^{\mathscr{E}^{\lambda}}_{\tau}}$. Hence, the map is surjective. Thus, ${\mathscr{E}}/{\mathcal{N}^{\mathscr{E}}_{\tau}}=\oplus_{\lambda\in\Lambda}\mathscr{E}^{\lambda}_{\tau} =\mathscr{E}_{\tau}$.
 \end{proof}
 \begin{corollary}
 	Let $\tau\in\mathcal{P}\left(\mathfrak{A}\right)$ and $\mathscr{E}=\mathit{l}_n^2(\mathfrak{A})$ for some $n$. Then $\mathscr{E}_{\tau}={\mathscr{E}}/{\mathcal{N}^{\mathscr{E}}_{\tau}}$. 
 \end{corollary}
 \begin{proof}
 	It follows immediately from Proposition \ref{oplus} and the fact that $\mathfrak{A}_{\tau}={\mathfrak{A}}/{\mathcal{N}^{\mathfrak{A}}_{\tau}}$.
 \end{proof}

Consider the $*$-representation $\phi_{\tau}:\mathfrak{A}\rightarrow\mathbb{B}\left(\mathfrak{A}_{\tau}\right)$. By \cite[Theorem 5.1.1]{mur}, there exists a unique vector $h_{\tau}\in\mathfrak{A}_{\tau}$ satisfying
\begin{equation}\label{sajjad1}
	\tau(a)=\left\langle h_{\tau},\phi_{\tau}(a)h_{\tau}\right\rangle \quad\left(a\in\mathfrak{A}\right).
\end{equation}
Moreover, $h_{\tau}$ is a unit vector and $\phi_{\tau}(a)h_{\tau}=a+{\mathcal{N}^{\mathfrak{A}}_{\tau}}$. We can construct the representation $\Phi_{\tau}:\mathscr{E}\rightarrow\mathbb{B}\left(\mathfrak{A}_{\tau},\mathcal{E}_{\tau}\right)$ by
\begin{align}
	\Phi_{\tau}(x)h=x\otimes h+\mathcal{M}_{\phi_{\tau}}\quad\left(x\in\mathscr{E}, h\in\mathfrak{A}_{\tau}\right). 
\end{align}
\begin{proposition}
	Let $\mathscr{E}$ be a Hilbert $\mathfrak{A}$-module. 
	\begin{enumerate}
		\item
		The map $T:\mathscr{E}_{\tau}\rightarrow\mathcal{E}_{\tau}$ defined by 
		$x+\mathcal{N}^{\mathscr{E}}_{\tau}\mapsto x\otimes h_{\tau}+\mathcal{M}_{\phi_{\tau}}$
		is isometric. Thus, $\mathscr{E}_{\tau}$ is a closed subspace of the Hilbert space $\mathcal{E}_{\tau}$.
		\item In the case when $\tau\in\mathcal{P}\left(\mathfrak{A}\right)$, it holds that $\mathcal{E}_{\tau}\cong\mathscr{E}_{\tau}$.
	\end{enumerate}
\end{proposition}
\begin{proof}
	(1)	We observe that 
	\begin{align*}
		\left\|x+\mathcal{N}^{\mathscr{E}}_{\tau}\right\|^2&=\left\langle	x+\mathcal{N}^{\mathscr{E}}_{\tau},	x+\mathcal{N}^{\mathscr{E}}_{\tau}\right\rangle=\tau(	\left\langle	x,	x\right\rangle)\\&=	\left\langle h_{\tau},\phi_{\tau}(\left\langle	x,	x\right\rangle)h_{\tau}\right\rangle\tag{by \eqref{sajjad1}}\\&=	\left\langle	 x\otimes h_{\tau}+\mathcal{M}_{\phi_{\tau}},	 x\otimes h_{\tau}+\mathcal{M}_{\phi_{\tau}}\right\rangle=\left\| x\otimes h_{\tau}+\mathcal{M}_{\phi_{\tau}}\right\|^2.
	\end{align*}
	(2) Let $\sum_{i=1}^{n}x_i\otimes h_i+\mathcal{M}_{\phi_{\tau}}\in\mathcal{E}_{\tau}$. Since $\tau$ is pure, we can write $h_i=a_i+\mathcal{N}^{\mathfrak{A}}_{\tau}$ for some $a_i\in\mathfrak{A}$ \cite[Corollary 1]{Kadison}. Then
	\begin{align*}
		\sum_{i=1}^{n}x_i\otimes h_i+\mathcal{M}_{\phi_{\tau}}&=	\sum_{i=1}^{n}x_i\otimes (a_i+\mathcal{N}^{\mathfrak{A}}_{\tau})+\mathcal{M}_{\phi_{\tau}}\\&=\sum_{i=1}^{n}x_i\otimes (\phi_{\tau}(a_i)h_{\tau})+\mathcal{M}_{\phi_{\tau}}\tag{since $\phi_{\tau}(a)h_{\tau}=a+{\mathcal{N}^{\mathfrak{A}}_{\tau}}$}\\&=\sum_{i=1}^{n}x_ia_i\otimes h_{\tau}+\mathcal{M}_{\phi_{\tau}}\tag{by the definition of $\mathcal{M}_{\phi_{\tau}}$}\\&=T\left(\sum_{i=1}^{n}x_ia_i+\mathcal{N}^{\mathscr{E}}_{\tau}\right).
	\end{align*}
	Therefore, the map $T$ is an isometric isomorphism. Hence $\mathcal{E}_{\tau}\cong\mathscr{E}_{\tau}$.
\end{proof}
Suppose that $\Phi:\mathscr{E}\rightarrow\mathbb{B}\left(\mathscr{H},\mathscr{K}\right) $ is a representation of $\mathscr{E}$. The commutant of $\Phi\left(\mathscr{E}\right) $ is defined as follows:
\begin{align*}
	\Phi\left(\mathscr{E}\right)^{\prime}=&\{ T_1\oplus T_2\in\mathbb{B}\left(\mathscr{H}\oplus\mathscr{K}\right): T_1\in\mathbb{B}\left(\mathscr{H}\right), T_2\in\mathbb{B}\left(\mathscr{K}\right),\\&T_2\Phi(x)=\Phi(x)T_1, T_1\Phi(x)^*=\Phi(x)^*T_2, x\in\mathscr{E}\}.
\end{align*}
\begin{remark}
Let $\tau\in\mathcal{P}\left(\mathfrak{A}\right)$.	It is easy to verify that 
$$\Phi_{\tau}(x)\left(a+{\mathcal{N}^{\mathfrak{A}}_{\tau}}\right)=xa+{\mathcal{N}^{\mathscr{E}}_{\tau}}$$
 and 
$$ \Phi_{\tau}(x)^*\left(y+{\mathcal{N}^{\mathscr{E}}_{\tau}}\right)=\left\langle x,y\right\rangle+{\mathcal{N}^{\mathfrak{A}}_{\tau}}.$$
It follows from \cite[Theorem 5.1.6]{mur} that $\phi_{\tau}$ is irreducible. Since $\Phi_{\tau}(\mathscr{E})\left({\mathfrak{A}}_{\tau}\right)$ is dense in ${\mathscr{E}}_{\tau}$. We infer from \cite[Proposition 3.6]{representation} that $\Phi_{\tau}$ is irreducible. Moreover, \cite[Proposition 4.5]{representation} ensures $\Phi_{\tau}(\mathscr{E})^{\prime}=\mathbb{C}\left(1_{\mathbb{B}\left({\mathfrak{A}}_{\tau} \right) }\oplus1_{\mathbb{B}\left({\mathscr{E}}_{\tau} \right) }\right)$.
\end{remark}
Let $\phi:\mathfrak{A}\rightarrow\mathfrak{B}$ be a $*$-homomorphism. By virtue of \cite[Theorem 3.1.6]{mur}, $\phi(\mathfrak{A})$ forms a $C^*$-subalgebra of $\mathfrak{B}$. The following proposition extends this result to the context of Hilbert $C^*$-modules.
\begin{proposition}\label{ideal}
 Let $\mathscr{E}$ be a Hilbert $\mathfrak{A}$-module and $\mathscr{F}$ be a Hilbert $\mathfrak{B}$-module. Suppose that $\phi:\mathfrak{A}\rightarrow\mathfrak{B}$ is a $*$-homomorphism. Let $\Phi:\mathscr{E}\rightarrow\mathscr{F}$ be a $\phi$-morphism. Then $\Phi\left(\mathscr{E}\right)$ is a Hilbert $\phi(\mathfrak{A})$-module As a consequence, it is closed in $\mathscr{F}$.
\end{proposition}
\begin{proof}
	It follows from $\left\langle\Phi(x),\Phi(y)\right\rangle=\phi\left(\left\langle x,y\right\rangle\right)$ that $\Phi\left(\mathscr{E}\right) $ is a pre-Hilbert $\phi(\mathfrak{A})$-module. We shall show that it is closed in $\mathscr{F}$. First, we show that $\ker\Phi=\mathscr{E}\ker\phi$. Let $\Phi(x)=0$. It follows from \cite[Lemma 4.4]{lance} that there exists $w\in\mathscr{E}$ such that $x=w\left|x\right|^{1/2}$. Since
	\begin{equation*}
		\phi(\left|x\right|^{1/2})=\phi\left(\left\langle x,x\right\rangle\right)^{1/4}=\left\langle\Phi(x),\Phi(x)\right\rangle^{1/4}=0,
	\end{equation*} 
we observe that $\ker\Phi\subseteq\mathscr{E}\ker\phi$. The other side is obvious. Now, it is deduced from \cite[Theorem 1.6]{ideal} that $\mathscr{E}/\ker\Phi$ equipped with the natural right $\mathfrak{A}/\ker\phi$-action and the inner product given by $\left\langle x+\ker\Phi,y+\ker\Phi\right\rangle=\phi\left(\left\langle x,y\right\rangle\right)$ is a Hilbert $\mathfrak{A}/\ker\phi$-module. Thus, the map 	$\mathscr{E}/\ker\Phi\rightarrow\mathscr{F}$ defined by
$x+\ker\Phi\mapsto\Phi(x)$ is isometric. It implies that $\Phi\left(\mathscr{E}\right)$ is closed in $\mathscr{F}$.
\end{proof}
\begin{corollary}\label{closed}
 Let $\tau\in\mathcal{S}\left(\mathfrak{A}\right)$. Then $\Phi_{\tau}\left(\mathscr{E}\right)$ is a Hilbert $\phi_{\tau}\left(\mathfrak{A}\right)$-module. Thus, $\Phi_{\tau}\left(\mathscr{E}\right)$ is closed in $\mathbb{B}\left(\mathfrak{A}_{\tau},\mathcal{E}_{\tau}\right)$. 
\end{corollary}
\begin{example}
	Suppose that $\mathscr{E}$ is a Hilbert $C^*$-module over a commutative $C^*$-algebra $\mathfrak{A}$ and $\tau\in\mathcal{P}\left(\mathfrak{A}\right)$. Then $\mathscr{E}_{\tau}={\mathscr{E}}/{\mathcal{N}^{\mathscr{E}}_{\tau}}$. To see this, let $x\in\mathscr{E}_{\tau}$. It follows that there exists a sequence $\left(x_n+\mathcal{N}^{\mathscr{E}}_{\tau}: n\in\mathbb{N}\right)$ converging to $x$. We can pick $a\in\mathfrak{A}$ such that $\tau(a)=1$. We observe that
	\begin{align*}
	\Phi_{\tau}\left(x_na^*\right)\left(b+\mathcal{N}^{\mathfrak{A}}_{\tau}\right)&= x_na^*b+\mathcal{N}^{\mathscr{E}}_{\tau}\\&
=\tau(a^*b)x_n+\mathcal{N}^{\mathscr{E}}_{\tau}
\tag{$\tau\in\Omega\left(\mathfrak{A}\right)$ is a $*$-homomorphism}\\&	=\theta_{x_n+\mathcal{N}^{\mathscr{E}}_{\tau},a+\mathcal{N}^{\mathfrak{A}}_{\tau}}\left(b+\mathcal{N}^{\mathfrak{A}}_{\tau}\right).
	\end{align*} 
Thus, $\Phi_{\tau}\left(x_na^*\right)=\theta_{x_n+\mathcal{N}^{\mathscr{E}}_{\tau},a+\mathcal{N}^{\mathfrak{A}}_{\tau}}$.
 It follows immediately that
 \begin{equation*}
\Phi_{\tau}\left(x_na^*\right)=\theta_{x_n+\mathcal{N}^{\mathscr{E}}_{\tau},a+\mathcal{N}^{\mathfrak{A}}_{\tau}}\rightarrow\theta_{x,a+\mathcal{N}^{\mathfrak{A}}_{\tau}},
 \end{equation*}
 as $n\to \infty$ in the norm topology. It follows from Corollary \ref{closed} that $\theta_{x,a+\mathcal{N}^{\mathfrak{A}}_{\tau}}\in\Phi_{\tau}\left(\mathscr{E}\right)$. Thus, there exists $y\in\mathscr{E}$ such that
 \begin{equation*}
 x=\theta_{x,a+\mathcal{N}^{\mathfrak{A}}_{\tau}}\left(a+\mathcal{N}^{\mathfrak{A}}_{\tau}\right)=	\Phi_{\tau}\left(y\right)\left(a+\mathcal{N}^{\mathfrak{A}}_{\tau}\right)=ya+\mathcal{N}^{\mathscr{E}}_{\tau}.
 \end{equation*}
Thus, $\mathscr{E}_{\tau}={\mathscr{E}}/{\mathcal{N}^{\mathscr{E}}_{\tau}}$.
\end{example}
Suppose that $\mathfrak{A}$ is a $C^*$-subalgebra of $\mathbb{B}\left(\mathscr{H}\right)$. Let $(\cdot)^\prime$ denote the commutant operation with respect to $\mathbb{B}\left(\mathscr{H}\right)$. Consider the universal representation $\phi_{u}:\mathfrak{A}\rightarrow\mathbb{B}\left(\mathscr{H}_u\right)$ and let $\mathfrak{A}_u=\phi_u(\mathfrak{A})$. The von Neumann algebra $\mathfrak{A}^{\prime\prime}_u$ is said to be the \emph{universal enveloping von Neumann algebra} of $\mathfrak{A}$. According to \cite[Theorem 3.1.3]{Manu}, $\mathfrak{A}^{\prime\prime}_u\cong\mathfrak{A}^{**}$, where $\mathfrak{A}^{**}$ is the bidual of $\mathfrak{A}$ equipped with $\sigma\left(\mathfrak{A}^{**},\mathfrak{A}^{*}\right)$-topology. It is important to note that if $\mathfrak{A}\subseteq\mathbb{B}\left(\mathscr{H}\right)$ and $\phi_{u}:\mathfrak{A}\rightarrow\mathbb{B}\left(\mathscr{H}_u\right)$ is the universal representation, then $\mathscr{H}$ can be viewed as a subspace of $\mathscr{H}_u$, and $\mathfrak{A}^{\prime\prime}\subseteq\mathfrak{A}^{\prime\prime}_u$. 

Let $\mathscr{E}$ be a Hilbert $\mathfrak{A}$-module. Consider the Hilbert $\mathfrak{A}^{**}$-module $\mathscr{E}^{\sharp}:=\mathscr{E}\otimes_{\iota}\mathfrak{A}^{**}$ as an extension of $\mathscr{E}$ by the algebra $\mathfrak{A}^{**}$, where $\iota:\mathfrak{A}\hookrightarrow\mathfrak{A}^{**}$ is the natural embedding. For $x_i,y_j\in\mathscr{E}$ and $a_i,b_j\in\mathfrak{A}^{**}$, we have
\begin{equation}\label{sharp}
	\left\langle\sum_{i=1}^{n}x_i\otimes a_i+\mathcal{M}_{\iota},\sum_{j=1}^{m}y_j\otimes b_j+\mathcal{M}_{\iota}\right\rangle=\sum_{i,j}	 a_i^*\left\langle x_i,y_j\right\rangle b_j,
\end{equation}
where $\mathcal{M}_{\iota}$ is defined in the paragraph preceding \eqref{msm1}. Let $1_u$ stand for the unit element of $\mathfrak{A}^{\prime\prime}_u\cong\mathfrak{A}^{**}$. We derive that the $\mathfrak{A}$-module map $\mathscr{E}\rightarrow\mathscr{E}^{\sharp}$
defined by $x\mapsto x\otimes1_{u}+\mathcal{M}_{\iota}$
is isometric. 
 \begin{proposition}\label{isimetric}
 	The map $\mathcal{L}\left(\mathscr{E}\right)\rightarrow\mathcal{L}\left(\mathscr{E}^{\sharp}\right) $ defined by $T\mapsto T_{\mathscr{E}^{\sharp}}$ is an isometric $*$-homomorphism.
 \end{proposition}
\begin{proof}
	 Following \eqref{T}, this map is a $*$-homomorphism. Let $S_{\mathscr{E}^{\sharp}}=0$. For all $x\in\mathscr{E}$, we have $Sx\otimes1_u+\mathcal{M}_{\o\iota}=S_{\mathscr{E}^{\sharp}}(x\otimes1_u+\mathcal{M}_{\o\iota})=0$. So $\left\langle Sx,Sx\right\rangle=\o\iota\left(\left\langle Sx,Sx\right\rangle \right)=0$. Thus, $S=0$. Then, \cite[Theorem 3.1.5]{mur} entails that the map is isometric.
\end{proof}
Next, let $\tau\in\mathcal{P}\left(\mathfrak{A}\right)$. Then, there exists a minimal projection $p_{\tau}\in\mathfrak{A}^{**}$, in the sense that $p_{\tau}\mathfrak{A}^{**}p_{\tau}=\mathbb{C}p_{\tau}$, such that
\begin{equation}\label{pedersen}
	\mathcal{N}^{\mathfrak{A}}_{\tau}=\left\lbrace x\in\mathfrak{A}: ap_{\tau}=0 \right\rbrace.
\end{equation}
In addition, the pure state $\tau$ has a unique extension to a normal state $\tau^{\prime}$ on $\mathfrak{A}^{**}$ satisfying $\tau^{\prime}\left(p_{\tau}\right)=1$; see \cite[Proposition 3.13.6]{Pedersen} and \cite[Lemma A]{Nourouzi}. 
\begin{lemma}\label{lem1}
	Let $\tau\in\mathcal{P}\left(\mathfrak{A}\right)$ and $\rho\in\mathcal{S}\left(\mathfrak{A}\right)$. If $\mathcal{N}^{\mathfrak{A}}_{\tau}=	\mathcal{N}^{\mathfrak{A}}_{\rho}$, then $\rho=\tau$.
\end{lemma}
\begin{proof}
	Utilizing \cite[Proposition 3.13.6]{Pedersen}, we get $\rho\in\mathcal{P}\left(\mathfrak{A}\right)$. It follows from \cite[Theorem 5.3.5]{mur} that $\rho=\tau$.
\end{proof}
Next, we extend \eqref{pedersen} to the framework of Hilbert $C^*$-modules.
\begin{proposition}\label{ped}
	Let $\mathscr{E}$ be a Hilbert $\mathfrak{A}$-module. If $\tau$ is a pure state on $\mathfrak{A}$, then there exists a minimal projection $p_{\tau}\in\mathfrak{A}^{**}$ satisfying \eqref{pedersen} such that
	\begin{enumerate}
		\item\label{ped1}
		$\mathcal{N}^{\mathscr{E}}_{\tau}=\left\lbrace x\in\mathscr{E}: xp_{\tau}=0 \right\rbrace$;
		\item
		the map $\mathscr{E}/\mathcal{N}^{\mathscr{E}}_{\tau}\rightarrow\mathscr{E}p_{\tau}$ defined by $x+\mathcal{N}^{\mathscr{E}}_{\tau}\mapsto xp_{\tau}$ is an isometric isomorphism.
	\end{enumerate}
\end{proposition}
\begin{proof}
(1)	Let $x\in\mathscr{E}$ such that $xp_{\tau}=0$. Therefore, $\left( \left|x\right|p_{\tau}\right)^*\left( \left|x\right|p_{\tau}\right)=\left\langle xp_{\tau},xp_{\tau}\right\rangle=0$, whence $\left|x\right|p=0$. It follows from \eqref{pedersen} that $\left|x\right|\in\mathcal{N}^{\mathfrak{A}}_{\tau}$. Hence, $\tau\left(\left\langle x,x\right\rangle\right)=0$. Thus, $\left\lbrace x\in\mathscr{E}: xp_{\tau}=0 \right\rbrace\subseteq\mathcal{N}^{\mathscr{E}}_{\tau}$.
		
		For the converse, let $x\in\mathscr{E}$ such that $\tau\left(\left\langle x,x\right\rangle\right)=0$. From \cite[Theorem 3.3.2]{mur} we conclude that $\tau\left(\left|x \right| \right)\leq\tau\left(\left\langle x,x\right\rangle\right)^{1/2}=0$. We deduce from \eqref{pedersen} that $\left|x\right|^{1/2}p_{\tau}=0$. It follows from \cite[Lemma 4.4]{lance} that there exists $w\in\mathscr{E}$ such that $x=w\left|x\right|^{1/2}$. Thus, $xp_{\tau}=w\left|x\right|^{1/2}p_{\tau}=0$. Hence, 	$\mathcal{N}^{\mathscr{E}}_{\tau}\subseteq\left\lbrace x\in\mathscr{E}: xp_{\tau}=0 \right\rbrace$.
	
	(2)	Let $\tau\in\mathcal{P}\left(\mathfrak{A}\right)$.  \cite[Theorem 5.1.13]{mur} implies that there exists a unique $\tilde{\tau}\in\mathcal{P}\left(\mathfrak{A}\oplus\mathbb{C}1_{u}\right)$, where $\mathfrak{A}\oplus\mathbb{C}1_{u}$ is the minimal unitization of $\mathfrak{A}$. Now, we can consider $\mathscr{E}$ as a Hilbert $\mathfrak{A}\oplus\mathbb{C}1_{u}$-module via the module action $x\left(a+\alpha\right)=xa+\alpha x$ for $x\in\mathscr{E}$, $a\in\mathfrak{A}$, and $\alpha\in\mathbb{C}$. From condition (\ref{ped1}), we derive that there exists a minimal projection $p_{\tau}\in\left( \mathfrak{A}\oplus\mathbb{C}1_u\right)^{**}=\left( \mathfrak{A}\oplus\mathbb{C}1_u\right)_{u}^{\prime\prime}=\mathfrak{A}_{u}^{\prime\prime}$ such that 	
		\begin{equation}\label{noruzi}
			\left\lbrace x\in\mathfrak{A}\oplus\mathbb{C}1_u: xp_{\tau}=0 \right\rbrace=\mathcal{N}^{\mathfrak{A}\oplus\mathbb{C}1_u}_{\tilde{\tau}}.
		\end{equation}
		Define $\rho:\mathfrak{A}\oplus\mathbb{C}1_u\rightarrow\mathbb{C}$ by $x\mapsto\lambda_x$, where 
		$p_{\tau}xp_{\tau}=\lambda_{x}p_{\tau}$. It is immediately deduced that $\rho\in\mathcal{S}\left(\mathfrak{A}\oplus\mathbb{C}1_u\right)$. It follows from \eqref{noruzi} that 	$\mathcal{N}^{\mathfrak{A}\oplus\mathbb{C}1_u}_{\tilde{\tau}}=	\mathcal{N}^{\mathfrak{A}\oplus\mathbb{C}1_u}_{\rho}$. We infer from Lemma \ref{lem1} that $\rho=\tilde{\tau}$. Thus, 
		\begin{align*}
			\left\|x+\mathcal{N}^{\mathscr{E}}_{\tau}\right\|&=\tau\left(\left\langle x,x \right\rangle\right)^{1/2}=\tilde{\tau}\left(\left\langle x,x \right\rangle\right)^{1/2}\\
			&=\rho\left(\left\langle x,x \right\rangle\right)^{1/2}=\left\| p_{\tau}\left\langle x,x \right\rangle p_{\tau}\right\| ^{1/2}=\left\|xp_{\tau}\right\|. 
		\end{align*}
			Clearly, The map $\mathscr{E}/\mathcal{N}^{\mathscr{E}}_{\tau}\rightarrow\mathscr{E}p_{\tau}$ is surjective. Hence, it is an isometric isomorphism.
\end{proof}
Inspired by the previous proposition, we can consider $\mathscr{E}_{\tau}=\mathscr{E}p_{\tau}$ as a Hilbert subspace of $\mathscr{E}^{\sharp}$. In what follows, we extensively investigate $\mathfrak{A}_{\tau}={\mathfrak{A}}/{\mathcal{N}^{\mathfrak{A}}_{\tau}}$ in the setting of Hilbert $C^*$-modules over a $C^*$-algebra of compact operators. First, we present the following lemma. 
\begin{lemma}\label{c_0}
	$\left( \Sigma_{\lambda\in\Lambda}\mathbb{K}\left(\mathscr{H}_{\lambda}\right)\right)^{**}\cong\prod_{\lambda\in\Lambda}\mathbb{B}\left(\mathscr{H}_{\lambda}\right)$.
	\end{lemma}
\begin{proof}
	It follows from \cite[p. 125]{mur} that $\mathbb{K}\left(\mathscr{H}_{\lambda}\right)^{**}\cong\mathbb{B}\left(\mathscr{H}_{\lambda}\right)$. Thus,  \eqref{duality} ensures that
		$\left( \Sigma_{\lambda\in\Lambda}\mathbb{K}\left(\mathscr{H}_{\lambda}\right)\right)^{**}\cong\left( \oplus^{\ell^1}_{\lambda\in\Lambda}\mathbb{K}\left(\mathscr{H}_{\lambda}\right)^*\right)^{*}\cong\prod_{\lambda\in\Lambda}\mathbb{B}\left(\mathscr{H}_{\lambda}\right)$.
\end{proof}
\begin{example}
	Let $\mathscr{E}$ be a Hilbert $C^*$-module over the $C^*$-algebra $\Sigma_{\lambda\in\Lambda}\mathbb{K}\left(\mathscr{H}_{\lambda}\right)$. Let $\tau\in\mathcal{P}\left(\Sigma_{\lambda\in\Lambda}\mathbb{K}\left(\mathscr{H}_{\lambda}\right)\right)$. By Proposition \ref{ped}, there exists a minimal projection $p_{\tau}\in\left( \Sigma_{\lambda\in\Lambda}\mathbb{K}\left(\mathscr{H}_{\lambda}\right)\right)^{**}$. It follows from Lemma \ref{c_0} that $p_{\tau}\in\prod_{\lambda\in\Lambda}\mathbb{B}\left(\mathscr{H}_{\lambda}\right)$. It is immediately deduced from \cite[p. 55]{mur} that $p_{\tau}=\theta_{e_{\lambda},e_{\lambda}}$, for some $\lambda\in\Lambda$ and unit vector $e_{\lambda}\in\mathscr{H}_{\lambda}$. Thus, $p_{\tau}\in \Sigma_{\lambda\in\Lambda}\mathbb{K}\left(\mathscr{H}_{\lambda}\right)$. Proposition \ref{ped}(2) entails that 
	$\mathscr{E}/\mathcal{N}^{\mathscr{E}}_{\tau}\cong\mathscr{E}p_{\tau}\subseteq\mathscr{E}$. Since $\mathscr{E}p_{\tau}$ is a Hilbert subspace of $\mathscr{E}$, we observe that 	$\mathscr{E}_{\tau}={\mathscr{E}}/{\mathcal{N}^{\mathscr{E}}_{\tau}}$. In other word, ${\mathscr{E}}/{\mathcal{N}^{\mathscr{E}}_{\tau}}$ is complete.
	
	Moreover, we show that there exist $\eta\in\Lambda$ and a unit vector $e_{\eta}\in\mathscr{H}_{\eta}$ such that 
	\begin{equation}
		\tau\left(\left(u_{\lambda}\right)\right)= \left\langle u_{\eta}e_{\eta},e_{\eta}\right\rangle_{\mathscr{H}_{\eta}}\quad\left(\left( u_{\lambda}\right)\in\Sigma_{\lambda\in\Lambda}\mathbb{K}\left(\mathscr{H}_{\lambda}\right)\right).
	\end{equation}
	To prove this claim, let $\left( u^{\eta}_{\lambda}\right)\in\Sigma_{\lambda\in\Lambda}\mathbb{K}\left(\mathscr{H}_{\lambda}\right)$ be defined as follows: 
	\begin{align*}
		u^{\eta}_{\lambda}=\begin{cases}
			u_{\eta} &\mbox{if}\quad \lambda=\eta\\0 &\mbox{if}\quad\lambda\neq\eta.
		\end{cases}	
	\end{align*}
	We can select $\eta\in\Lambda$ such that $\tau\left(\left( u^{\eta}_{\lambda}\right)\right)\neq0$. Define $\rho:\Sigma_{\lambda\in\Lambda}\mathbb{K}\left(\mathscr{H}_{\lambda}\right)\rightarrow\mathbb{C}$ by $\rho\left(\left(u_{\lambda}\right) \right)=\tau\left(\left( u^{\eta}_{\lambda}\right)\right)$. Clearly, $\rho$ is a positive linear functional satisfying $\rho\leq\tau$. This ensures that $\rho=t\tau$ for some $t\in[0,1]$. Since $\rho\left(\left( u^{\eta}_{\lambda}\right)\right)=\tau\left(\left( u^{\eta}_{\lambda}\right)\right)\neq0$, we derive that $\tau=\rho$. According to \cite[Example 5.1.1]{mur}, there exists a unit vector $e_{\eta}\in\mathscr{H}_{\eta}$ such that 
	$\tau\left(\left(u_{\lambda}\right)\right)= \left\langle u_{\eta}e_{\eta},e_{\eta}\right\rangle_{\mathscr{H}_{\eta}}$.
\end{example}
We conclude this section with the following open problem:
\begin{problem}
	Let $\mathscr{E}$ be a Hilbert $\mathfrak{A}$-module. Is $\mathscr{E}_{\tau}={\mathscr{E}}/{\mathcal{N}^{\mathscr{E}}_{\tau}}$ for each $\tau\in \mathcal{P}\left(\mathfrak{A}\right)$? Moreover, can one find some classes of Hilbert $C^*$-modules for which the equality holds?
\end{problem}
\section{Extensions of Schatten norms}
Let $\mathscr{H}$ and $\mathscr{K}$ be Hilbert spaces. For $T\in\mathbb{B}\left(\mathscr{H},\mathscr{K}\right)$ and $k\geq1$, the \emph{Schatten $k$-norm} of $T$ is defined by $\left\| T\right\|_{\left( k\right) }:=\left(\mathrm {Tr}~ \left| T\right| ^{k}\right)^{{1}/{k}}$. We denote the space of $k$-Schatten operators by $\mathcal{L}^k\left(\mathscr{H},\mathscr{K}\right):=\left\lbrace T\in\mathbb{B}\left(\mathscr{H},\mathscr{K}\right): \left\| T\right\|_{\left( k\right)}<\infty \right\rbrace$. Kittaneh \cite{Kittaneh5} investigated some inequalities for Schatten $k$-norms. In \cite[p. 15]{sajjad}, the definition of $\left\|\cdot\right\|_{\left[ k\right]}$, introduced by Stern and van Suijlekom \cite[Theorem 3.28]{sche}, is generalized in the framework of all Hilbert $C^*$-modules. In this context, for $T\in\mathcal{L}(\mathscr{E})$, we define 
$$\left\|T\right\|_{\left[ k\right]}:=\sup_{\tau\in\mathcal{P}\left(\mathfrak{A}\right)}\left\|T_{\mathscr{E}_{\tau}}\right\|_{\left(k\right)},$$
where $T_{\mathscr{E}_{\tau}}$ is constructed according to \eqref{T}. We demonstrated that $\left\|T\right\|_{\left[k\right]}$ is a norm on $\mathcal{L}^k(\mathscr{E}):=\left\lbrace T\in\mathcal{L}(\mathscr{E}):\left\|T\right\|_{\left[ k\right]}<\infty \right\rbrace $ with the expected properties of Schatten $k$-norms. 
 
The concept of an orthonormal basis is crucial in the standard definition of Schatten $k$-norms. Stern and van Suijlekom \cite{sche} introduced Schatten classes for countably generated Hilbert $C^*$-modules over commutative $C^*$-algebras using the concept of a frame. However, it is possible that a Hilbert $C^*$-module may not have a frame or an orthonormal basis, which limits their adaptability for extension to general Hilbert $C^*$-modules. Nevertheless, this limitation can be explored by utilizing absolutely summing operators; refer to \cite{2012} for more details.

In \cite[Definition 4.1]{sajjad}, the \emph{trace class operators} and the \emph{Hilbert--Schmidt operators} were extended in the context of Hilbert $C^*$-modules. These classes are denoted by $\left( \tilde{\Pi}_1\left(\mathscr{E},\mathscr{F}\right),\tilde{\pi}_1(\cdot)\right)$ and $\left( \tilde{\Pi}_2\left(\mathscr{E},\mathscr{F}\right),\tilde{\pi}_2(\cdot)\right)$, respectively. Furthermore,  we obtain a further extension of the Schatten $k$-norms by replacing $\mu_n$ with $\mu^{\mathscr{E}}_n$ and allowing $k \geq 1$ in \cite[Definition~3.3]{sajjad}.
\begin{definition}\label{pp}
	Let $\mathscr{E}$ and $\mathscr{F}$ be Hilbert $\mathfrak{A}$-modules. For $T \in\mathcal{L}(\mathscr{E},\mathscr{F})$, we define
	\begin{equation}\label{1}
		\pi^{\mathscr{E},\mathscr{F}}_k(T): = \sup\left\lbrace \left\|\sum_{i=1}^n \left\langle\left|T\right|^kx_i,x_i\right\rangle\right\|^{1/k}: \left( x_1,\ldots,x_n\right) \in\left( \mathscr{E}^n,\mu^{\mathscr{E}}_n\right)_{[1]}, n\in\mathbb{N}\right\rbrace.
	\end{equation}	
\end{definition}
Notice that $\left|T\right|=\left(T^*T\right)^{1/2}\in\mathcal{L}(\mathscr{E})$. The set of all adjointable operators with $	\pi^{\mathscr{E},\mathscr{F}}_k(T) < \infty$ is denoted by ${\Pi}_{k}(\mathscr{E},\mathscr{F})$. We write ${\pi}^{\mathscr{E}}_k(\cdot)$ and ${\Pi}_k(\mathscr{E})$ for ${\pi}^{\mathscr{E},\mathscr{E}}_k(\cdot)$ and ${\Pi}_k(\mathscr{E},\mathscr{E})$, respectively.
\begin{proposition}
Suppose that $\mathscr{E}$, $\mathscr{F}$, and $\mathscr{G}$ are Hilbert $\mathfrak{A}$-modules. The following statements hold.
\begin{enumerate}
	\item
	$\left({\Pi}_{2}(\mathscr{E},\mathscr{F}),{\pi}^{\mathscr{E},\mathscr{F}}_{2}(\cdot)\right) $ and 	$\left({\Pi}_{1}(\mathscr{E}),{\pi}^{\mathscr{E}}_{1}(\cdot)\right) $ are Banach spaces. 
	\item
	For $T\in \mathcal{L}(\mathscr{E},\mathscr{F})$ and $S\in \mathcal{L}(\mathscr{F},\mathscr{G})$, it holds that ${\pi}^{\mathscr{E},\mathscr{G}}_2(ST)\leq \left\|S\right\|{\pi}^{\mathscr{E},\mathscr{F}}_2(T)$ and ${\pi}^{\mathscr{E},\mathscr{G}}_2(ST)\leq {\pi}^{\mathscr{F},\mathscr{G}}_2(S)\left\|T\right\|$.
		\item
	For $T,S\in \mathcal{L}(\mathscr{E})$, we have
	$\left\|T\right\|\leq{\pi}^{\mathscr{E}}_1(T)$, ${\pi}^{\mathscr{E}}_1(TS)\leq \left\|T\right\|{\pi}^{\mathscr{E}}_1(S)$ and $\left({\pi}^{\mathscr{E}}_2(T)\right) ^2\leq\left\|T \right\|{\pi}^{\mathscr{E}}_1(T)$.
	\item 
	For $T\in\mathcal{L}(\mathscr{E},\mathscr{F})$, where $	\mathscr{E}=\ker(T)\oplus\overline{\rm{ran}(T^*)}$ and $	\mathscr{F}=\ker(T^*)\oplus\overline{\rm{ran}(T)}$, it holds that ${\pi}^{\mathscr{E},\mathscr{F}}_2(T)={\pi}^{\mathscr{E},\mathscr{F}}_2(T^*).$
	\item 
	For $T,S\in\mathcal{L}(\mathscr{E})$, where $	\mathscr{E}=\ker(T)\oplus\overline{\rm{ran}(T^*)}=\ker(T^*)\oplus\overline{\rm{ran}(T)}$, it holds that ${\pi}^{\mathscr{E}}_1(T)={\pi}^{\mathscr{E}}_1(T^*)$ and ${\pi}^{\mathscr{E}}_1(TS)\leq\left\|S\right\|{\pi}^{\mathscr{E}}_1(T).$ 
\end{enumerate}
\end{proposition}
\begin{proof}
	The proofs are similar to those in \cite[Section 4]{sajjad}, so we do not include them.
\end{proof}
\begin{remark}\label{remark}
Let $\mathscr{E}$ and $\mathscr{F}$ be Hilbert $C^*$-modules over a commutative $C^*$-algebra $\mathfrak{A}$. For $k\geq1$,
$\left({\Pi}_{k}(\mathscr{E},\mathscr{F}),{\pi}^{\mathscr{E},\mathscr{F}}_{k}(\cdot)\right)$ is a Banach space.	To see this, let us define the map $\left\|\cdot\right\|_{\mathit{l}^k_n(\mathscr{F})}: \mathscr{F}^n\rightarrow\mathbb{R}^+$ by
$\left\|\left(x_1,\dots,x_n\right)\right\|_{\mathit{l}^k_n(\mathscr{F})}=\left\|\sum_{i=1}^{n}\left|x_i\right|^k \right\|^{{1}/{k}}$ for $n\in\mathbb{N}$.
It follows from commutativity of $\mathfrak{A}$ that the sequence $\left( \left( \mathscr{F}^n,\left\|\cdot\right\|_{\mathit{l}^2_n(\mathscr{F})}\right) :n\in\mathbb{N}\right)$ is a power-normed space. In view of \cite[p. 11]{sajjad}, ${\Pi}_{k}(\mathscr{E},\mathscr{F})=\mathcal{ML}(\mathscr{E},\mathscr{F})$ and ${\pi}^{\mathscr{E},\mathscr{F}}_{k}(T)=\left\|T\right\|_{mb}$ for $T\in{\Pi}_{k}(\mathscr{E},\mathscr{F})$. Here, $\left( \mathcal{ML}(\mathscr{E},\mathscr{F}),\left\|T\right\|_{mb}\right)$ is the space of adjointable multi-bounded operators calculated with respect to the power-normed spaces $\left(\left( \mathscr{E}^n,\mu^\mathscr{E}_n\right):n\in\mathbb{N}\right) $ and $\left( \left( \mathscr{F}^n,\left\|\cdot\right\|_{\mathit{l}^k_n(\mathscr{F})}\right) :n\in\mathbb{N}\right)$. It follows from \cite[p. 5]{sajjad} that $\left({\Pi}_{k}(\mathscr{E},\mathscr{F}),{\pi}^{\mathscr{E},\mathscr{F}}_{k}(\cdot)\right)$ is a Banach space. The reader may consult \cite{sajjad} for undefined notations.
\end{remark}
In fact, in the case when $\mathfrak{A}$ is noncommutative, we are not sure if ${\pi}^{\mathscr{E}}_{k}(\cdot)$ forms a norm on ${\Pi}_{k}(\mathscr{E})$ for any $k\geq1$. Furthermore, some Schatten properties may not hold true in $\left({\Pi}_{k}(\mathscr{E}),{\pi}^{\mathscr{E}}_{k}(\cdot)\right)$. However, we are able to establish some interesting Schatten properties. 
\begin{proposition}
	Let $\mathscr{E}$ be a Hilbert $\mathfrak{A}$-module. Let $n\in\mathbb{N}$ and $x_1,\dots,x_n\in\mathscr{E}$. Then
	\begin{equation}\label{mu}
		\mu^{\mathscr{E}}_{n}(x_1,\dots,x_n)=\mu^{\mathscr{E}^{\sharp}}_{n}(x_1,\dots,x_n).
	\end{equation} 
\end{proposition}
\begin{proof}
	Since the $\mathfrak{A}$-module map $\mathscr{E}\rightarrow\mathscr{E}^{\sharp}$
	defined by $x\mapsto x\otimes1_{u}+\mathcal{M}_{\iota}$
	is isometric, we arrive at $\mu^{\mathscr{E}}_{n}(x_1,\dots,x_n)\leq\mu^{\mathscr{E}^{\sharp}}_{n}(x_1,\dots,x_n)$.
	
	Conversely, let us suppose that $1_{u}\in\mathfrak{A}$. Pick
	\begin{equation}\label{alpha}
		\alpha\in\left\lbrace\lambda>0: \sum_{i=1}^{n} \left\langle y,x_i\right\rangle\left\langle x_i,y\right\rangle\leq\lambda^2\left\langle y,y\right\rangle \textrm{~for~all~} y\in\mathscr{E}\right\rbrace.
	\end{equation}
	 Suppose that $$\sum_{j=1}^{m}y_j\otimes a_j+\mathcal{M}_{\iota}\in\frac{\mathscr{E}\otimes\mathfrak{A}^{**}}{\mathcal{M}_{\iota}}.$$
	
	 By using \cite[Lemma 4.1.4]{mur} the $C^*$-algebra $\mathfrak{A}$ is dense in $\mathfrak{A}^{**}\cong\mathfrak{A}_{u}^{\prime\prime}$ with respect to the strong topology on $\mathbb{B}\left(\mathscr{H}_{u}\right)$. Thus, for each $1\leq i\leq n$, there exists a net $\left(a^{\lambda}_i\right)_\lambda$ in $\mathfrak{A}$ such that $a^{\lambda}_i\rightarrow a_i$ strongly. It follows from \eqref{alpha} that
		$$\sum_{i=1}^{n} \left\langle \sum_{j=1}^{m}y_ja^{\lambda}_j,x_i\right\rangle\left\langle x_i,\sum_{j=1}^{m}y_ja^{\lambda}_j\right\rangle\leq\alpha^2\left\langle \sum_{j=1}^{m}y_ja^{\lambda}_j,\sum_{j=1}^{m}y_ja^{\lambda}_j\right\rangle$$
		or, equivalently,
		$$\sum_{i=1}^{n} \sum_{r,s=1}^{m}\left( a^{\lambda}_r\right)^* \left\langle y_r,x_i\right\rangle\left\langle x_i,y_s\right\rangle a^{\lambda}_s\leq\alpha^2\left(\sum_{r,s=1}^{m} \left( a^{\lambda}_r\right)^* \left\langle y_r,y_s\right\rangle a^{\lambda}_s\right).$$
	Hence, for every $e\in\mathscr{H}_{u}$, we have
	$$\sum_{i=1}^{n} \sum_{r,s=1}^{m}\left\langle\left( a^{\lambda}_r\right)^* \left\langle y_r,x_i\right\rangle\left\langle x_i,y_s\right\rangle a^{\lambda}_se,e\right\rangle\leq\alpha^2\sum_{r,s=1}^{m} \left\langle\left( a^{\lambda}_r\right)^* \left\langle y_r,y_s\right\rangle a^{\lambda}_se,e\right\rangle$$
	or, equivalently,
	$$\sum_{i=1}^{n} \sum_{r,s=1}^{m}\left\langle \left\langle y_r,x_i\right\rangle\left\langle x_i,y_s\right\rangle a^{\lambda}_se,a^{\lambda}_re\right\rangle\leq\alpha^2\sum_{r,s=1}^{m} \left\langle \left\langle y_r,y_s\right\rangle a^{\lambda}_se,a^{\lambda}_re\right\rangle. $$
	
	By taking limits, we get 
	$$\sum_{i=1}^{n} \sum_{r,s=1}^{m}\left\langle \left\langle y_r,x_i\right\rangle\left\langle x_i,y_s\right\rangle a_se,a_re\right\rangle\leq\alpha^2\sum_{r,s=1}^{m} \left\langle \left\langle y_r,y_s\right\rangle a_se,a_re\right\rangle,$$
	which is equivalent to
		$$\sum_{i=1}^{n} \sum_{r,s=1}^{m}\left\langle\left( a_r\right)^* \left\langle y_r,x_i\right\rangle\left\langle x_i,y_s\right\rangle a_se,e\right\rangle\leq\alpha^2\sum_{r,s=1}^{m} \left\langle\left( a_r\right)^* \left\langle y_r,y_s\right\rangle a_se,e\right\rangle$$
		and
		\begin{align*}
		&\sum_{i=1}^{n} \left\langle \sum_{j=1}^{m}y_j\otimes a_j+\mathcal{M}_{\iota},x_i\right\rangle\left\langle x_i,\sum_{j=1}^{m}y_j\otimes a_j+\mathcal{M}_{\iota}\right\rangle\\& \qquad\leq\alpha^2\left\langle \sum_{j=1}^{m}y_j\otimes a_j+\mathcal{M}_{\iota},\sum_{j=1}^{m}y_j\otimes a_j+\mathcal{M}_{\iota}\right\rangle.
	\end{align*}
	
	Since $\left(\mathscr{E}\otimes\mathfrak{A}^{**}\right)/\mathcal{M}_{\iota}$ is dense in $\mathscr{E}^{\sharp}$, we derive that $\sum_{i=1}^{n} \left\langle z,x_i\right\rangle\left\langle x_i,z\right\rangle\leq\alpha^2\left\langle z,z\right\rangle$ for all $z\in\mathscr{E}^{\sharp}$. Thus, 
	$$		\alpha\in\left\lbrace\lambda>0: \sum_{i=1}^{n} \left\langle z,x_i\right\rangle\left\langle x_i,z\right\rangle\leq\lambda^2\left\langle z,z\right\rangle \textrm{~for~all~} z\in\mathscr{E}^{\sharp}\right\rbrace.$$
	Therefore,
	\begin{align*}
		&\mu^{\mathscr{E}^{\sharp}}_{n}(x_1,\dots,x_n)=\min\left\lbrace\lambda>0: \sum_{i=1}^{n} \left\langle z,x_i\right\rangle\left\langle x_i,z\right\rangle\leq\lambda^2\left\langle z,z\right\rangle \textrm{~for~all~} z\in\mathscr{E}^{\sharp}\right\rbrace\\&\leq\min\left\lbrace\lambda>0: \sum_{i=1}^{n} \left\langle y,x_i\right\rangle\left\langle x_i,y\right\rangle\leq\lambda^2\left\langle y,y\right\rangle \textrm{~for~all~} y\in\mathscr{E}\right\rbrace=\mu^{\mathscr{E}}_{n}(x_1,\dots,x_n).
	\end{align*}
 In the case where $1_{u}\notin\mathfrak{A}$, we can use $\mathfrak{A}\oplus\mathbb{C}1_{u}$ instead of $\mathfrak{A}$ in virtue of \cite[p. 9]{sajjad}.
\end{proof}
\begin{corollary}\label{important}
	${\pi}^{\mathscr{E}}_k(T)\leq{\pi}^{\mathscr{E}^{\sharp}}_k(T_{\mathscr{E}^{\sharp}})$ for each $T\in\mathcal{L}(\mathscr{E})$ and $k \geq 1$. 
\end{corollary}
Our next theorem refines Corollary \ref{important}. To establish it, we need some lemmas.
\begin{lemma}\label{Lemma1}
 Suppose that $\tau\in\mathcal{S}(\mathfrak{A})$ and $k\geq1$. Then $\tau(a)^k\leq\tau(a^k)$ for all $a\geq 0$.
\end{lemma} 
\begin{proof}
	First, let $\mathfrak{A}$ be unital. The $C^*$-subalgebra of $\mathfrak{A}$ generated by $a$ and $1$ can be identified with $\mathcal{C}\left(\rm{sp}(a)\right)$, where $\rm{sp}(a)$ denotes the spectrum of $a$. The positive element $a$ is identified by the inclusion map $z\in\mathcal{C}\left(\rm{sp}(a)\right)$. Clearly, $1,z\in\mathcal{C}\left(\rm{sp}(a),\mathbb{R}\right)$, the space of real valued functions in $\mathcal{C}\left(\rm{sp}(a)\right)$. The restriction of $\tau$ on $\mathcal{C}\left(\rm{sp}(a),\mathbb{R}\right)$ is also a state. Therefore,
\begin{equation}
	\tau(a)=\tau|_{\mathcal{C}\left(\rm{sp}(a),\mathbb{R}\right)}(z)=\int_{\rm{sp}(a)}z d\nu,
\end{equation}
where $\nu$ is a measure on $\rm{sp}(a)$ with $\nu(\rm{sp}(a))=1$ by the Riesz--Kakutani theorem. Let $t\geq1$ such that $\frac{1}{k}+\frac{1}{t}=1$. Employing the H\"{o}lder inequality, we have 
\begin{align*}
&\int_{\rm{sp}(a)}zd\nu\leq\left( \int_{\rm{sp}(a)}z^k d\nu\right)^{1/k} \left( \int_{\rm{sp}(a)} d\nu\right)^{1/t}\leq\left( \int_{\rm{sp}(a)}z^k d\nu\right)^{1/k}\\&=\left( \tau|_{\mathcal{C}\left(\rm{sp}(a),\mathbb{R}\right)}(z^k)\right)^{1/k}=\left( \tau(a^k)\right)^{1/k}.
\end{align*}
Thus, we have $\tau(a)^k\leq\tau(a^k)$. In the case when $\mathfrak{A}$ is not unital, we use the minimal unitization $\mathfrak{A}\oplus \mathbb{C}$ of $\mathfrak{A}$ and extend $\tau$ to $\mathfrak{A}\oplus\mathbb{C}$; see \cite[Theorem 3.3.9]{mur}.
\end{proof}

The family of all mutually orthogonal $n$-tuples $\left(x_1,\dots,x_n\right)\subseteq\mathit{l}_n^2(\mathscr{E})$, where $\left\langle x_i,x_i\right\rangle$ is a projection for $1\leq i\leq n$, is denoted by $\mathcal{D}^{\mathscr{E}}_n$. The norm closed convex hull of $\mathcal{D}^{\mathscr{E}}_n$ is denoted by $\overline{\rm co}(\mathcal{D}^{\mathscr{E}}_n)$. By a Hilbert subspace of $\mathscr{E}$, we mean a subspace of $\mathscr{E}$ that is a Hilbert space under the same norm as $\mathscr{E}$.
The next lemma is stated as follows.
\begin{lemma}\label{dim}
	 Suppose that $\mathscr{E}$ is a Hilbert $\mathfrak{A}$-module and $\mathscr{H}$ is a Hilbert subspace of $\mathscr{E}$. 	Let $n\in \mathbb{N}$ and $x_1,\dots,x_n\in\mathscr{H}$. Then 
	 \begin{equation}\label{Hilbert}
	 	\mu^{\mathscr{H}}_{n}(x_1,\dots,x_n)=\mu^{\mathscr{E}}_{n}(x_1,\dots,x_n).
	 \end{equation}
\end{lemma}
\begin{proof}
	First, let us assume that $\dim(\mathscr{H})\geq n$. Pick $\left( y_1,\dots,y_n\right) \in\mathcal{D}^{\mathscr{H}}_n$. It follows from \cite[Lemma 2.4]{sajjad2} that $\mu^{\mathscr{H}}_{n}(y_1,\dots,y_n)=\mu^{\mathscr{E}}_{n}(y_1,\dots,y_n)=\max_{1\leq i\leq n}\left\|y_i\right\|= 1$. From \cite[lemma 3.3]{sajjad2} we deduce that $$\left( \mathscr{H}^n,\mu^{\mathscr{H}}_n\right)_{[1]}=\overline{\rm co}(\mathcal{D}^{\mathscr{H}}_n)\subseteq\left( \mathscr{E}^n,\mu^{\mathscr{E}}_n\right)_{[1]}.$$ 
	This implies that $\mu^{\mathscr{H}}_{n}(x_1,\dots,x_n)\geq\mu^{\mathscr{E}}_{n}(x_1,\dots,x_n)$. The reverse inequality is evident by \eqref{mue}.
	
	 In the case when $\dim(\mathscr{H})=0$, we have $\mathscr{H}=0$ and \eqref{Hilbert} is evident. Now, assume that $1\leq\dim(\mathscr{H})<n$. Then
	 \begin{align*}
	 	\mu^{\mathscr{H}}_{n}(x_1,\dots,x_n)&=\sup\left\lbrace\left\| \sum_{i=1}^{n}x_i\alpha_i\right\| : \alpha_i\in\mathbb{C},\sum_{i=1}^{n}\left|\alpha_i\right|^2\leq 1\right\rbrace\tag{by \eqref{Lemma3.7}}
	 	\\&=\mu^{\mathit{l}_n^2(\mathscr{H})}_{n}(x_1,\dots,x_n)\tag{since $\mathit{l}_n^2(\mathscr{H})$ is a Hilbert space}\\&=\mu^{\mathit{l}_n^2(\mathscr{E})}_{n}(x_1,\dots,x_n)\tag{since $\mathit{l}_n^2(\mathscr{H})\subseteq\mathit{l}_n^2(\mathscr{E})$ and $\dim\left(\mathit{l}_n^2(\mathscr{H})\right)\geq n$}\\&=\sup\left\lbrace\left\| \sum_{i=1}^{n}x_ia_i\right\| : a_i\in\mathfrak{A},\left\|\sum_{i=1}^{n}a^*_ia_i\right\|\leq 1\right\rbrace\tag{by \eqref{Lemma3.7}}\\&=\mu^{\mathscr{E}}_{n}(x_1,\dots,x_n).
	 \end{align*}
\end{proof}
\begin{proposition}\label{Jaegermann}
		Let $T\in\mathbb{B}(\mathscr{H})$ and $k\geq1$. Then 
	${\pi}^{\mathscr{H}}_k(T)=\left\|T\right\|_{\left( k\right)}$.
\end{proposition}
\begin{proof}
	 Jaegermann \cite[Proposition 10.1]{Jaegermann} stated that $\pi^{\mathscr{H}}_2(S)=\left\|S\right\|_{(2)}$ for $S\in\mathbb{B}(\mathscr{H})$. Thus,
\begin{equation*}
\pi^{\mathscr{H}}_k(T)=\pi^{\mathscr{H}}_2(\left|T\right|^\frac{k}{2})^\frac{2}{k}=\left\|\left|T\right|^\frac{k}{2}\right\|_{(2)}^{\frac{2}{k}}=\left(\mathrm {Tr}~ \left| T\right| ^{k}\right)^{\frac{1}{k}}=\left\|T\right\|_{\left( k\right)}.
\end{equation*}
\end{proof}
Our following main result investigates the extension of Proposition \ref{Jaegermann} to the setting of Hilbert $C^*$-modules.It can be interpreted as a result that provides upper and lower bounds for the norm $\left\|T\right\|_{\left[k\right]}$.
\begin{theorem}\label{thm1}
	Let $T\in\mathcal{L}(\mathscr{E})$ and $k\geq1$. Then 
 $${\pi}^{\mathscr{E}}_k(T)\leq\left\|T\right\|_{\left[k\right]}\leq{\pi}^{\mathscr{E}^{\sharp}}_k(T_{\mathscr{E}^{\sharp}})$$.
\end{theorem}
\begin{proof}
Proof of the first inequality: Pick $n\in\mathbb{N}$ and $x_1,\ldots,x_n\in\mathscr{E}$ such that $\mu^{\mathscr{E}}_n(x_1,\ldots,x_n)\leq1$. It follows from \cite[Theorem 5.1.11]{mur} that there exists $\tau\in\mathcal{P}\left(\mathfrak{A}\right)$ such that 
$$\tau\left(\left( \sum_{i=1}^{n}\left\langle\left|T\right|^kx_i,x_i \right\rangle\right)^{\frac{1}{k}} \right)=\left\|\left( \sum_{i=1}^{n}\left\langle\left|T\right|^kx_i,x_i \right\rangle\right)^{\frac{1}{k}} \right\|.$$ 	
Since $\mathscr{E}_{\tau}$ is the Hilbert completion of $\mathscr{E}/\mathcal{N}^{\mathscr{E}}_{\tau}$, we infer from \eqref{Proposition 3.6} that
	\begin{align*}
		&\mu^{\mathscr{E}_{\tau}}_n(x_1+\mathcal{N}^{\mathscr{E}}_{\tau},\ldots,x_n+\mathcal{N}^{\mathscr{E}}_{\tau})\\&=\min\left\lbrace \lambda>0: \sum_{i=1}^{n}\left| \left\langle x_i+\mathcal{N}^{\mathscr{E}}_{\tau},x+\mathcal{N}^{\mathscr{E}}_{\tau}\right\rangle\right|^2\leq \lambda^2\left| x+\mathcal{N}^{\mathscr{E}}_{\tau}\right|^2 \textrm{~for~all~} x\in\mathscr{E}\right\rbrace \\&=\min\left\lbrace \lambda>0: \sum_{i=1}^{n}\left| \tau\left( \left\langle x_i,x\right\rangle\right)\right|^2\leq \lambda^2\tau\left( \left\langle x,x\right\rangle\right) \textrm{~for~all~} x\in\mathscr{E}\right\rbrace \\&\leq\min\left\lbrace \lambda>0: \tau\left(\sum_{i=1}^{n}\left| \left\langle x_i,x\right\rangle\right|^2\right)\leq \lambda^2\tau\left( \left\langle x,x\right\rangle\right) \textrm{~for~all~} x\in\mathscr{E}\right\rbrace \tag{by Lemma \ref{Lemma1}}\\&\leq\min\left\lbrace \lambda>0: \sum_{i=1}^{n}\left| \left\langle x_i,x\right\rangle\right|^2\leq \lambda^2 \left\langle x,x\right\rangle \textrm{~for~all~} x\in\mathscr{E}\right\rbrace \\& =\mu^{\mathscr{E}}_n(x_1,\ldots,x_n)\leq 1.
	\end{align*}
 We have
	\begin{align*}
		\left\|\left( \sum_{i=1}^{n}\left\langle\left|T\right|^kx_i,x_i \right\rangle\right)^{\frac{1}{k}} \right\|&=\tau\left(\left( \sum_{i=1}^{n}\left\langle\left|T\right|^kx_i,x_i \right\rangle\right) ^{\frac{1}{k}}\right)\\&\leq\tau\left( \sum_{i=1}^{n}\left\langle\left|T\right|^kx_i,x_i \right\rangle\right)^{\frac{1}{k}}\tag{by Lemma \ref{Lemma1}}\\&
		=\left( \sum_{i=1}^{n} \left\langle \left( \left|T\right|^k\right)_{\mathscr{E}_\tau} \left( x_i+\mathcal{N}^{\mathscr{E}}_{\tau}\right) ,x_i+\mathcal{N}^{\mathscr{E}}_{\tau}\right\rangle\right)^{\frac{1}{k}} \\&=\left( \sum_{i=1}^{n} \left\langle \left| T_{\mathscr{E}_\tau}\right|^k \left( x_i+\mathcal{N}^{\mathscr{E}}_{\tau}\right) ,x_i+\mathcal{N}^{\mathscr{E}}_{\tau}\right\rangle\right)^{\frac{1}{k}}\\&\leq{\pi}^{\mathscr{E}_\tau}_k(T_{\mathscr{E}_\tau})= \left\|T_{\mathscr{E}_\tau}\right\|_{\left(k\right)}\tag{by Proposition \ref{Jaegermann}}.
	\end{align*}
	Hence ${\pi}^{\mathscr{E}}_k(T)\leq\sup_{\tau\in\mathcal{P}\left( \mathfrak{A}\right)}\left\|T_{\mathscr{E}_\tau}\right\| _{\left(k\right)}=\left\|T\right\|_{\left[ k\right]}$.

Proof of the second inequality: Let $\tau\in\mathcal{P}\left(\mathfrak{A}\right)$. By Lemma \ref{ped}, there exists a minimal projection $p_{\tau}\in\mathfrak{A}^{**}$ such that $\mathscr{E}/\mathcal{N}^{\mathscr{E}}_{\tau}\cong\mathscr{E}p_{\tau}$. We deduce from Proposition \ref{Jaegermann} that
	\begin{align*}
		&\left\|T_{\mathscr{E}_\tau}\right\| _{\left(k\right)}= \sup\left\lbrace \left\|\sum_{i=1}^n \left\langle\left|T_{\mathscr{E}_\tau}\right|^ky_i,y_i\right\rangle\right\|^{1/k}: \left( y_1,\ldots,y_n\right) \in\left( \mathscr{E}_{\tau}^n,\mu^{\mathscr{E}_{\tau}}_n\right)_{[1]}, n\in\mathbb{N}\right\rbrace\\&=\sup\left\lbrace \left\|\sum_{i=1}^n \left\langle\left|T_{\mathscr{E}_\tau}\right|^ky_i,y_i\right\rangle\right\|^{1/k}: \left( y_1,\ldots,y_n\right) \in\left( \left(\overline{\mathscr{E}p_{\tau}}\right) ^n,\mu^{\overline{\mathscr{E}p_{\tau}}}_n\right)_{[1]}, n\in\mathbb{N}\right\rbrace\tag{since $\mathscr{E}_{\tau}=\overline{\mathscr{E}p_{\tau}}$ }\\&=\sup\left\lbrace \left\|\sum_{i=1}^n \left\langle\left|T_{\mathscr{E}_\tau}\right|^kx_ip_{\tau},x_ip_{\tau}\right\rangle\right\|^{1/k}: \mu_n^{\overline{\mathscr{E}p_{\tau}}}\left( x_1p_{\tau},\ldots,x_np_{\tau}\right)\leq1, x_i\in\mathscr{E}, n\in\mathbb{N}\right\rbrace\\&=\sup\left\lbrace \left\|\sum_{i=1}^n \left\langle\left|T_{\mathscr{E}^{\sharp}}\right|^kx_ip_{\tau},x_ip_{\tau}\right\rangle\right\|^{1/k}: \mu_n^{\mathscr{E}^{\sharp}}\left( x_1p_{\tau},\ldots,x_np_{\tau}\right)\leq1, x_i\in\mathscr{E}, n\in\mathbb{N}\right\rbrace\tag{by Lemma \ref{dim}}.
	\end{align*}
	Hence $\left\|T_{\mathscr{E}_\tau}\right\| _{\left(k\right)}\leq{\pi}^{\mathscr{E}^{\sharp}}_k(T_{\mathscr{E}^{\sharp}})$.	Therefore, $\left\|T\right\|_{\left[k\right]}\leq{\pi}^{\mathscr{E}^{\sharp}}_k(T_{\mathscr{E}^{\sharp}})$.
\end{proof}
\begin{remark}
Suppose that $\mathscr{E}$ is a Hilbert $C^*$-module over a $C^*$-algebra $\mathfrak{A}$ and $\mathfrak{A}$ is reflexive. It is known that $\mathfrak{A}$ must be finite-dimensional. Therefore, $\mathfrak{A}$ is of the form $\prod_{1\leq i\leq m}\mathbb{M}_{n_i}(\mathbb{C})=\Sigma_{1\leq i\leq m}\mathbb{M}_{n_i}(\mathbb{C})$.
 In this case, $\mathscr{E}^{\sharp}=\mathscr{E}\otimes_{\iota}\mathfrak{A}^{**}=\mathscr{E}\otimes_{\iota}\mathfrak{A}=\mathscr{E}$. Thus, Proposition \ref{isimetric} implies that the map $\mathcal{L}\left(\mathscr{E}\right)\rightarrow\mathcal{L}\left(\mathscr{E}^{\sharp}\right) $ defined by $T\mapsto T_{\mathscr{E}^{\sharp}}$ is an isometric $*$-isomorphism. Theorem \ref{thm1} ensures that ${\pi}^{\mathscr{E}}_k(T)\leq\left\|T\right\|_{\left[k\right]}\leq{\pi}^{\mathscr{E}^{\sharp}}_k(T_{\mathscr{E}^{\sharp}})={\pi}^{\mathscr{E}}_k(T)$. Therefore, ${\pi}^{\mathscr{E}}_k(T)=\left\|T\right\|_{\left[k\right]}$
 for all $k\geq1$. 
\end{remark}
Note that $\mathbb{K}\left(\mathscr{H}\right)^{\sharp}=\mathbb{B}\left(\mathscr{H}\right)\neq\mathbb{K}\left(\mathscr{H}\right)$ in the case where $\mathscr{H}$ is an infinite-dimensional Hilbert space. Thus, the condition $\mathscr{E}^{\sharp}=\mathscr{E}$ is not necessary to establish ${\pi}^{\mathscr{E}}_k(T)=\left\|T\right\|_{\left[k\right]}$ as illustrated in the following example for Hilbert $C^*$-modules over $C^*$-algebras of compact operators. To present this example, we need a lemma.

\begin{lemma}\label{minimal}
	Let $u\in\prod_{\lambda\in\Lambda}\mathbb{B}\left(\mathscr{H}_{\lambda}\right)$. Then $\left\|u\right\|=\sup\left\|pup\right\|$, where the supremum is taken over all minimal projections $p\in\prod_{\lambda\in\Lambda}\mathbb{B}\left(\mathscr{H}_{\lambda}\right)$. 
\end{lemma}
\begin{proof}
	Note that $p=\theta_{e_{\lambda},e_{\lambda}}$ for some $\lambda\in\Lambda$ and unit vector $e_{\lambda}\in \mathscr{H}_{\lambda}$.
\end{proof}
\begin{example}\label{exam1}
	Let $\mathscr{E}$ be a Hilbert $C^*$-module over a $C^*$-algebra of the form $\Sigma_{\lambda\in\Lambda}\mathbb{K}\left(\mathscr{H}_{\lambda}\right)$. We prove that ${\pi}^{\mathscr{E}}_k(T)={\pi}^{\mathscr{E}^{\sharp}}_k(T_{\mathscr{E}^{\sharp}})$ for each $T\in\mathcal{L}(\mathscr{E})$ and $k \geq 1$. By virtue of Corollary \ref{important}, it is enough to show that ${\pi}^{\mathscr{E}}_k(T)\geq{\pi}^{\mathscr{E}^{\sharp}}_k(T_{\mathscr{E}^{\sharp}})$. Pick $\varepsilon>0$, there exists $\left( \sum_{j=1}^{m}x^{(i)}_j\otimes a^{(i)}_j\right)_i\in\mathscr{E}^{\sharp}$ such that
$$\mu_n^{\mathscr{E}^{\sharp}}\left(\sum_{j=1}^{m}x^{(1)}_j\otimes a^{(1)}_j,\dots,\sum_{j=1}^{m}x^{(n)}_j\otimes a^{(n)}_j\right)\leq 1$$
and
$$ {\pi}^{\mathscr{E}^{\sharp}}_k(T_{\mathscr{E}^{\sharp}})-\frac{\varepsilon}{2}\leq\left\| \left( \sum_{i=1}^{n} \left\langle \left| T_{\mathscr{E}^{\sharp}}\right|^k \left( \sum_{j=1}^{m}x^{(i)}_j\otimes a^{(i)}_j\right),\left( \sum_{j=1}^{m}x^{(i)}_j\otimes a^{(i)}_j\right)\right\rangle\right)^{\frac{1}{k}}\right\|.$$
	 We derive from Lemma \ref{minimal} that there exists a minimal projection $p\in\prod_{\lambda\in\Lambda}\mathbb{B}\left(\mathscr{H}_{\lambda}\right)$ such that 
		\begin{align*}
	&\hspace{-1cm}\left\| \left( \sum_{i=1}^{n} \left\langle \left| T_{\mathscr{E}^{\sharp}}\right|^k \left( \sum_{j=1}^{m}x^{(i)}_j\otimes a^{(i)}_jp\right),\left( \sum_{j=1}^{m}x^{(i)}_j\otimes a^{(i)}_jp\right)\right\rangle\right)^{\frac{1}{k}}\right\|\\&=\left\| p\left( \sum_{i=1}^{n} \left\langle \left| T_{\mathscr{E}^{\sharp}}\right|^k \left( \sum_{j=1}^{m}x^{(i)}_j\otimes a^{(i)}_j\right),\left( \sum_{j=1}^{m}x^{(i)}_j\otimes a^{(i)}_j\right)\right\rangle\right)^{\frac{1}{k}}p\right\| \\&\geq\left\|\left( \sum_{i=1}^{n} \left\langle \left| T_{\mathscr{E}^{\sharp}}\right|^k \left( \sum_{j=1}^{m}x^{(i)}_j\otimes a^{(i)}_j\right),\left( \sum_{j=1}^{m}x^{(i)}_j\otimes a^{(i)}_j\right)\right\rangle\right)^{\frac{1}{k}}\right\|-\frac{\varepsilon}{2}\\
&\geq{\pi}^{\mathscr{E}^{\sharp}}_k(T_{\mathscr{E}^{\sharp}})-{\varepsilon}.
	\end{align*}
	Since $\Sigma_{\lambda\in\Lambda}\mathbb{K}\left(\mathscr{H}_{\lambda}\right)$ is an ideal of $\prod_{\lambda\in\Lambda}\mathbb{B}\left(\mathscr{H}_{\lambda}\right)$, we observe that $a^{(i)}_jp\in\Sigma_{\lambda\in\Lambda}\mathbb{K}\left(\mathscr{H}_{\lambda}\right)$. Thus, we identify $\sum_{j=1}^{m}x^{(i)}_j\otimes a^{(i)}_jp$ with $\sum_{j=1}^{m}x^{(i)}_ja^{(i)}_jp$. Moreover, by using Proposition \ref{mu},
		\begin{align*}
	&\mu_n^{\mathscr{E}}\left(\sum_{j=1}^{m}x^{(1)}_ja^{(1)}_jp,\dots,\sum_{j=1}^{m}x^{(n)}_j a^{(n)}_jp\right)=	\mu_n^{\mathscr{E}^{\sharp}}\left(\sum_{j=1}^{m}x^{(1)}_j\otimes a^{(1)}_jp,\dots,\sum_{j=1}^{m}x^{(n)}_j\otimes a^{(n)}_jp\right)\\&\leq\mu_n^{\mathscr{E}^{\sharp}}\left(\sum_{j=1}^{m}x^{(1)}_j\otimes a^{(1)}_j,\dots,\sum_{j=1}^{m}x^{(n)}_j \otimes a^{(n)}_j\right)\leq 1\tag{by using \eqref{Proposition 3.4(1)}}.
	\end{align*}
Hence, we deduce that ${\pi}^{\mathscr{E}}_k(T)\geq{\pi}^{\mathscr{E}^{\sharp}}_k(T_{\mathscr{E}^{\sharp}})-\varepsilon$. Therefore, ${\pi}^{\mathscr{E}}_k(T)={\pi}^{\mathscr{E}^{\sharp}}_k(T_{\mathscr{E}^{\sharp}})$.
\end{example}
Now, we recall some definitions from \cite[Definition 2.3]{sche}.
 In the context of Banach $\mathfrak{A}$-modules, a countable subset $\mathbb{E}$ of $\mathscr{E}$ is said to be \emph{a set of generators of the Hilbert $C^*$-module} $\mathscr{E}$ if the $\mathfrak{A}$-linear span of $\mathbb{E}$ is norm-dense in $\mathscr{E}$. 

We say that a sequence $\left( f_i: i\in\mathbb{N}\right) $ in $\mathscr{E}$ is a standard normalized frame or simply a \emph{frame} if 
$$\left\langle x,y\right\rangle=\sum_{i=1}^{\infty}\left\langle x,f_i\right\rangle\left\langle f_i,y\right\rangle$$
in the norm topology for each $x,y\in\mathscr{E}$; see \cite[p. 21]{comm}. In fact, a Hilbert $C^*$-module $\mathscr{E}$ need not always possess a frame, but it is known that a countably generated Hilbert $C^*$-module $\mathscr{E}$ over a commutative $C^*$-algebra $\mathfrak{A}$ possesses a frame; see \cite{ASAD}.

\begin{lemma}\label{frame}
	Let $\mathscr{E}$ be a Hilbert $\mathfrak{A}$-module that has a frame $\left( f_i: i\in\mathbb{N}\right)$. Then $\left( f_i: i\in\mathbb{N}\right)$ is a frame for $\mathscr{E}^{\sharp}$.
\end{lemma}
\begin{proof}
	Let	$\sum_{j=1}^{m}x_j\otimes a_j+\mathcal{M}_{\iota},\sum_{j=1}^{m}y_j\otimes b_j+\mathcal{M}_{\iota} \in\left( {\mathscr{E}\otimes\mathfrak{A}^{**}}\right) /{\mathcal{M}_{\iota}}$. It follows that
	\begin{align*}
		&\left\langle\sum_{j=1}^{m}x_j\otimes a_j+\mathcal{M}_{\iota},\sum_{j=1}^{m}y_j\otimes b_j+\mathcal{M}_{\iota}\right\rangle=\sum_{r,s=1}^{m}a^*_r \left\langle x_r,y_s\right\rangle b_s\\&=\sum_{r,s=1}^{m}a^*_r \sum_{i=1}^{\infty}\left\langle x_r,f_i\right\rangle\left\langle f_i,y_s\right\rangle b_s\tag{$x_r,y_s\in\mathscr{E}$}\\&=\sum_{i=1}^{\infty}\left\langle \sum_{j=1}^{m}x_j\otimes a_j+\mathcal{M}_{\iota},f_i\right\rangle\left\langle f_i,\sum_{j=1}^{m}y_j\otimes b_j+\mathcal{M}_{\iota}\right\rangle.	
	\end{align*}
	Since $\left( {\mathscr{E}\otimes\mathfrak{A}^{**}}\right) /{\mathcal{M}_{\iota}}$ is norm-dense in $\mathscr{E}^{\sharp}$, a standrad argument shows that $\left( f_i: i\in\mathbb{N}\right)$ is a frame for $\mathscr{E}^{\sharp}$.
\end{proof}
\begin{lemma}\label{commutative}
	 Let $\mathscr{E}$ be a Hilbert $C^*$-module over a commutative $C^*$-algebra possessing a frame $\left( f_i: i\in\mathbb{N}\right)$. For $k\geq1$, the series $\sum_{i=1}^{\infty}\left\langle \left|T \right|^k f_i,f_i\right\rangle$ converges weakly in $\mathfrak{A}^{**}$ if and only if $T\in \Pi_k(\mathscr{E})$. Further, $${\pi}_k^{\mathscr{E}}(T)=\left\| \sum_{i=1}^{\infty}\left\langle \left|T \right|^kf_i,f_i\right\rangle\right\|^{1/k}.$$
\end{lemma}
\begin{proof}
 We derive from \cite[Theorem 4.6]{sajjad} that
 \begin{align*}
 	{\pi}_k^{\mathscr{E}}(T)={\pi}_2^{\mathscr{E}}\left( \left| T\right|^{\frac{k}{2}}\right)^{\frac{2}{k}}=\left( \left\| \sum_{i=1}^{\infty}\left\langle \left|T \right|^kf_i,f_i\right\rangle\right\|^{\frac{1}{2}}\right)^{\frac{2}{k}} =\left\| \sum_{i=1}^{\infty}\left\langle \left|T \right|^kf_i,f_i\right\rangle\right\|^{\frac{1}{k}}.
 \end{align*}	
\end{proof}
\begin{example}\label{exam2}
	Let $\mathscr{E}$ be a Hilbert $C^*$-module over a commutative $C^*$-algebra $\mathfrak{A}$ that has a frame $\left( f_i: i\in\mathbb{N}\right)$. It is easy to verify that $\mathfrak{A}^{**}\cong\mathfrak{A}_{u}^{\prime\prime}$ is also commutative. Then Lemma \ref{frame} implies that $\mathscr{E}^{\sharp}$ is a Hilbert $C^*$-module over a commutative $C^*$-algebra $\mathfrak{A}^{**}$ possessing the frame $\left( f_i: i\in\mathbb{N}\right)$. Consequently, Lemma \ref{commutative} ensures that
	\begin{equation}
{\pi}^{\mathscr{E}}_k(T)=\left\| \sum_{i=1}^{\infty}\left\langle \left|T\right|^k f_i,f_i\right\rangle\right\|^{1/k} =\left\| \sum_{i=1}^{\infty}\left\langle \left|T_{\mathscr{E}^{\sharp}}\right|^k f_i,f_i\right\rangle\right\|^{1/k}={\pi}^{\mathscr{E}^{\sharp}}_k(T_{\mathscr{E}^{\sharp}})
	\end{equation} 
for each $T\in \mathcal{L}\left(\mathscr{E}\right)$. 
\end{example}
	 In the setting of Examples \ref{exam1} and \ref{exam2} and Theorem \ref{thm1}, we infer that ${\pi}^{\mathscr{E}}_k(T)=\left\|T\right\|_{\left[k\right]}={\pi}^{\mathscr{E}^{\sharp}}_k(T)$ for each $T\in \mathcal{L}\left(\mathscr{E}\right)$ . Thus, $\left( {\Pi}_k\left(\mathscr{E}\right),{\pi}^{\mathscr{E}}_k\right)$ enjoys the properties of the norm $\left\|\cdot\right\|_{\left[k\right]}$.  For instance, we show that some Schatten $k$-norm inequalities may hold for $\left( {\Pi}_k\left(\mathscr{E}\right),{\pi}^{\mathscr{E}}_k(\cdot)\right)$. We gather some of them in the following proposition.
\begin{proposition}
	Let $\mathscr{E}$ be a Hilbert $\mathfrak{A}$-module such that	${\pi}^{\mathscr{E}}_k(T)={\pi}^{\mathscr{E}^{\sharp}}_k(T_{\mathscr{E}^{\sharp}})$ for each $T\in \mathcal{L}\left(\mathscr{E}\right)$. Then for each $k\geq1$,
 \begin{enumerate}
 	\item
 		$\left( {\Pi}_k\left(\mathscr{E}\right),{\pi}^{\mathscr{E}}_k(\cdot)\right)$ is a normed space.
 		\item
 	(H\"{o}lder inequality)
 	$${\pi}^{\mathscr{E}}_{k}\left(TS\right)\leq{\pi}^{\mathscr{E}}_{r}\left(T\right){\pi}^{\mathscr{E}}_{r^{\prime}}\left(S\right)$$ for $T,S\in\mathcal{L}\left(\mathscr{E}\right)$ and $r,r^{\prime}\geq1$ satisfying $\frac{1}{r}+\frac{1}{r^{\prime}}=\frac{1}{k}$.
 	\item
$${\pi}^{\mathscr{E}}_{2k}\left(\left|T\right|-\left|S\right|\right)\leq{\pi}^{\mathscr{E}}_{2k}\left(T+S\right)^{1/2}{\pi}^{\mathscr{E}}_{2k}\left(T-S\right)^{1/2}$$ for $T,S\in\mathcal{L}\left(\mathscr{E}\right)$; see \cite[Theorem 2.1]{Kittaneh5}.
 	\item
 $$\alpha{\pi}^{\mathscr{E}}_{k}\left(\left|T\right|-\left|S\right| \right)\leq{\pi}^{\mathscr{E}}_{2k}\left(T+S\right){\pi}^{\mathscr{E}}_{2k}\left(T-S\right)$$ for $T,S\in\mathcal{L}\left(\mathscr{E}\right)$ satisfying $\left|T\right|+\left|S\right|\geq\alpha\geq0$; see \cite[Theorem 2.2]{Kittaneh5}.
 	\item
 	$${\pi}^{\mathscr{E}}_{k}\left(T^rX+XS^r\right)\leq r\alpha^{r-1}{\pi}^{\mathscr{E}}_{k}\left(TX+XS\right)$$ for $T,S,X\in\mathcal{L}\left(\mathscr{E}\right)$ with $T,S\geq\alpha>0$ and $0<r\leq1$; see \cite[Corollary 3.2]{Kittaneh5}.
 	\end{enumerate}
\end{proposition}
\begin{proof}
	Notice that ${\pi}^{\mathscr{E}}_k(T)=\left\|T\right\|_{[k]}=\sup_{\tau\in\mathcal{P}\left(\mathfrak{A}\right)}\left\|T_{\mathscr{E}_\tau}\right\|_{\left(k\right)}$ and we know that the map $\mathcal{L}\left(\mathscr{E}\right)\rightarrow\mathcal{L}\left(\mathscr{E}_{\tau}\right)$ defined by $T\mapsto T_{\mathscr{E}_\tau}$ is a $*$-homomorphism.
\end{proof} 
We conclude this paper with the following open problem:
\begin{problem}
	Let $\mathscr{E}$ be a Hilbert $\mathfrak{A}$-module. Is ${\pi}^{\mathscr{E}}_k(T)={\pi}^{\mathscr{E}^{\sharp}}_k(T_{\mathscr{E}^{\sharp}})$ for each $T\in \mathcal{L}\left(\mathscr{E}\right)$? Moreover, can one find some classes of Hilbert $C^*$-modules for which the equality holds?
\end{problem}
\medskip
\noindent \textit{Acknowledgment.} The authors would like to sincerely thank the referee for valuable comments improving the paper.
\medskip
\noindent \textit{Conflict of Interest Statement.} On behalf of the authors, the corresponding author states that there is no conflict of interest.\\
\medskip
\noindent\textit{Data Availability Statement.} Data sharing not applicable to this article as no datasets were generated or analyzed during the current study.
\medskip
\bibliographystyle{amsplain}

\end{document}